\documentclass{article} 
\usepackage[usenames,dvipsnames]{color}
\usepackage{amsmath} 
\usepackage{amssymb} 
\usepackage{bm}
\usepackage{amsthm}
\usepackage{CJK}
\usepackage{indentfirst}
\usepackage{tikz}
\usepackage{amsmath,amsfonts,amsthm,mathrsfs}
\usepackage{hyperref}
\numberwithin{equation}{section}
\usepackage{geometry}

\usepackage{esint} 

\numberwithin{equation}{section}
\newtheorem{theorem}{Theorem}
\newtheorem{lemma}{Lemma}

\newtheorem{corollary}{Corollary}
\newtheorem{definition}{Definition}
\newtheorem{remark}{Remark}

\usepackage{algorithm,algpseudocode}
\usepackage{amsthm}

\def\NN{\mathbb{N}}
\def\PP{\mathbb{P}}
\def\RR{\mathbb{R}}

\def\YY{\mathbb{Y}}

\def\SS{\mathbb{S}}

\def\mal{\max\limits}
\def\mil{\min\limits}
\def\sul{\sum\limits}

\def\ss{ s}

\newcommand{\lrt}[1]{\left({#1}\right)}

\def\tQ{\widetilde{Q}}

\title{Condition Numbers and Eigenvalue Spectra\\ of Shallow  Networks on Spheres}
\author{Xinliang Liu\thanks{Ocean University of China, Qingdao 266100, China}\and Tong Mao\thanks{King Abdullah University of Science and Technology, Thuwal 23955, Saudi Arabia}\and Jinchao Xu\footnotemark[2]}
\date{}

\begin{document}
	
\maketitle

\begin{abstract}
    We present an estimation of the condition numbers of the \emph{mass} and \emph{stiffness} matrices arising from shallow ReLU$^k$ neural networks defined on the unit sphere~$\mathbb{S}^d$.
    In particular, when $\{\theta_j^*\}_{j=1}^n \subset \mathbb{S}^d$ is \emph{antipodally quasi-uniform}, the condition number is sharp. Indeed, in this case, we obtain sharp asymptotic estimates for the full spectrum of eigenvalues and characterize the structure of the corresponding eigenspaces, showing that the smallest eigenvalues are associated with an eigenbasis of low-degree polynomials while the largest eigenvalues are linked to high-degree polynomials. This spectral analysis establishes a precise correspondence between the approximation power of the network and its numerical stability. 
\end{abstract}

\section{Introduction}

Neural networks have become a cornerstone of modern machine learning and artificial intelligence, with widespread applications across fields such as computer vision~\cite{krizhevsky2012imagenet}, natural language processing~\cite{vaswani2017attention}, scientific computing~\cite{raissi2019physics}, healthcare~\cite{esteva2017dermatologist}, and finance~\cite{heaton2017deep}. Their remarkable empirical success has motivated a growing body of theoretical work aimed at understanding their expressive power, approximation capabilities, optimization landscape, and generalization behavior.

From a theoretical perspective, it has long been known that shallow neural networks with non-polynomial activation functions possess the universal approximation property, provided the network is sufficiently wide~\cite{cybenko1989approximation,hornik1989multilayer,leshno1993multilayer}. Building upon this foundational insight, researchers have sought to quantify the rate of approximation in various function spaces. In particular, Barron spaces~\cite{barron1993universal,barron1994approximation,maiorov2000stochastic,klusowski2018approximation}, Sobolev spaces~\cite{mhaskar1993approximation,pinkus1999approximation,yang2024optimal}, and Besov spaces~\cite{siegel2023optimal} have been studied extensively in the context of shallow networks.

Among the various activation functions studied in the literature, the Rectified Linear Unit (ReLU) and its higher-order variants $\sigma_k(x) := \max(0,x)^k$ have received particular attention due to their simplicity, nonlinearity, and homogeneity. A standard shallow ReLU$^k$ neural network with $n$ neurons takes the form
\[
\Sigma_n^k := \left\{ f(x) = \sum_{j=1}^n a_j \sigma_k(\theta_j \cdot x + b_j) : a_j \in \RR,\, (\theta_j, b_j) \in \RR^{d} \times \RR \right\},
\]
and is known to approximate functions in Barron or Sobolev spaces with rates depending on the regularity of the target function~\cite{klusowski2018approximation,siegel2023optimal,yang2024optimal}. For example, in the context of Barron-type spaces, it has been shown that ReLU$^k$ networks can achieve the optimal approximation rate $\mathcal{O}(n^{-\frac{d + 2k + 1}{2d}})$, for functions in $\mathcal{B}^k(\Omega)$, the Barron space associated with the activation $\sigma_k$~\cite{siegel2023optimal}.

More recently, we proposed a linearized variant of shallow ReLU$^k$ networks \cite{liu2025achieving}, in which the inner parameters $\{\theta_j^*\}_{j=1}^n \subset \mathbb{S}^d$ are fixed, and only the outer coefficients are optimized. The resulting function space, denoted by
\[
L_n^k := \mathrm{span}\left\{ \sigma_k(\theta_j^* \cdot x) \right\}_{j=1}^n,
\]
is a linear subspace of $\Sigma_n^k$, and admits a convex least-squares training procedure. Despite this simplification, we proved that $L_n^k$ achieves the same approximation rate as $\Sigma_n^k$ over Sobolev spaces:
\[
\inf_{f_n \in L_n^k} \|f - f_n\|_{L^2(\Omega)} \lesssim n^{-\frac{d + 2k + 1}{2d}}, \qquad \text{for } f \in \mathcal H^{\frac{d + 2k + 1}{2}}(\Omega).
\]
The linear structure of $L_n^k$ not only avoids the non-convexity inherent in training $\Sigma_n^k$, but also permits refined analysis of stability and conditioning—topics largely inaccessible in the nonlinear setting.

While the approximation properties of shallow ReLU$^k$ networks are now well understood, a critical aspect that remains largely unexplored is their \emph{numerical stability}. In numerical linear algebra, the stability and accuracy of solving a linear system $A x = b$ depend crucially on the \emph{condition number} of the coefficient matrix,
$$\kappa(A) := \|A\|_2 \, \|A^{-1}\|_2,$$
which measures the sensitivity of the solution $x$ to perturbations in the data or to roundoff errors~\cite{trefethen1997numerical,golub2013matrix,saad2003iterative}. When the condition number $\kappa(A)$ is well controlled, numerical solvers can achieve faster convergence, maintain high accuracy, and remain robust against data perturbations and floating-point roundoff errors~\cite{axelsson1994iterative,benzi2002preconditioning}. Ensuring good conditioning is therefore a central objective in classical numerical methods such as finite elements~\cite{ciarlet1978fem,xu1992iterative}.

In the context of shallow neural networks, the least-squares formulation leads to symmetric positive definite Gram matrices, whose conditioning governs both the stability and efficiency of training. Two such matrices play a central role in our analysis:
\begin{equation}
    M = \left( \int_{\SS^d} \sigma_k(\theta_i^* \cdot \omega) \, \sigma_k(\theta_j^* \cdot \omega) \, d\omega \right)_{i,j=1}^n,
\end{equation}
referred to as the \emph{mass matrix}, involving $\mathcal L^2$ inner products of the basis functions, and
\begin{equation}
    K_s = \left( \int_{\SS^d} (\nabla_\omega^s \sigma_k(\theta_i^* \cdot \omega)) \cdot (\nabla_\omega^s \sigma_k(\theta_j^* \cdot \omega)) \, d\omega \right)_{i,j=1}^n,
\end{equation}
referred to as the \emph{stiffness matrix}, involving inner products of their gradients. The condition numbers $\kappa(M)$ and $\kappa(K_s)$ quantify the spectral stability of $L_n^k$ in $\mathcal L^2$ and $\mathcal H^1$ norms, respectively, and directly impact the robustness and efficiency of any computational procedure built upon this basis.

These matrices are not merely theoretical constructs; they arise naturally in fundamental computational tasks. The mass matrix $M$ is the Gram matrix for the basis functions in $\mathcal L^2$. In a standard least-squares regression, the problem reduces to solving the normal equations $J^\top J a = J^\top y$, where the Jacobian $J$ contains the basis functions evaluated at the data points. The matrix $J^\top J$ is the empirical mass matrix, which converges to $M$ as the data becomes dense. The stiffness matrix $K_s$ appears when using variational methods to solve $2s$-order elliptic PDEs. Furthermore, in the context of fully nonlinear networks, the $s=1$ stiffness matrix $K_1$ is closely related to the Neural Tangent Kernel (NTK)\cite{jacot2018ntk}, which is formed from the Jacobian of the network output with respect to the inner-layer parameters.

In classical FEM, the (properly scaled) mass matrix is well conditioned, while the stiffness matrix exhibits polynomial growth in the degrees of freedom (e.g., $\kappa \simeq h^{-2}\simeq n^{2/d}$ on quasi-uniform meshes), motivating preconditioning strategies~\cite{ciarlet1978fem,xu1992iterative}. In machine learning, a useful spectral perspective is provided by the \emph{neural tangent kernel} (NTK)~\cite{jacot2018ntk}: in the infinite-width limit, training dynamics linearize and reduce to kernel regression with a certain kernel. The spectrum of NTK and its variants has been analyzed extensively for different architectures and activation functions, including ReLU, erf, and smooth activations~\cite{bordelon2020spectrum,liu2020ntk,canatar2021spectral,basri2022spectral}. For example, for the ReLU‐activated ResNTK on the hypersphere $\SS^{d}$, the eigenfunctions are spherical harmonics and the eigenvalues decay polynomially as $\lambda_n\simeq n^{-(d+1)/d}$~\cite{basri2022spectral}, and related work~\cite{zhang2023shallow} extended this analysis to the $d$-dimensional ball, obtaining the same decay rate. This spectral decay is closely related to the phenomenon of  \emph{frequency principle} \cite{xu2019frequency} and \emph{spectral bias} \cite{rahaman2019spectral}, where networks prioritize learning low-frequency functions. Building on the connection to finite element methods, it has been argued that this bias is an intrinsic property of the ReLU activation and can be mitigated by using activations based on B-splines, leading to improved conditioning and faster training~\cite{hong2022activation}.  The conditioning of Gram matrices in machine learning has also been studied from different perspectives. For instance, in the context of shallow Ritz methods for solving PDEs, similar ill-conditioning has been identified in the 1D case, where the condition number of the Gram matrix is shown to grow polynomially with the network width \cite{cai2024efficient}. Our work provides a precise, high-dimensional generalization of this phenomenon. From a numerical linear algebra viewpoint, the rapid decay of the smallest eigenvalue, which we characterize, is closely related to the concept of \emph{numerical rank}, impacting training stability and generalization \cite{yang2024effective,tang2025structured}. By providing sharp asymptotic estimates for the full spectrum, our results offer a rigorous theoretical foundation for these empirical and numerical observations.

A continuous–discrete correspondence links the spectrum of a continuous kernel operator to that of its empirical Gram matrix. In particular, spectral convergence theory for empirical kernel matrices asserts that, under mild regularity and i.i.d.\ sampling, the $j$-th eigenvalues of $n^{-1}M$ and $n^{-1}K^s$ converge to those of the integral operators with kernels
$$q(\theta,\eta)=\int_{\SS^d} \sigma_k(\theta \cdot \omega) \, \sigma_k(\eta \cdot \omega)\,d\omega,\quad
q_s(\theta,\eta)=\int_{\SS^d} (\nabla_\omega^s\sigma_k(\theta \cdot \omega)) \cdot (\nabla_\omega^s\sigma_k(\eta \cdot \omega))\,d\omega$$
as $n \to \infty$~\cite{koltchinskii2000random,rosasco2010operator,elkaroui2010spectrum}. Moreover, numerical experiments in~\cite[Fig.~10]{zhang2023shallow} compute the spectrum of the \emph{discrete} Gram matrix and reveal strong attenuation of high-frequency components due to finite-precision effects. However, no theoretical results for the eigenvalues of the discrete Gram matrix are provided there, which is precisely the focus of the present paper in the context of the mass and stiffness matrices $M$ and $K_\ss$ arising from fixed-parameter ReLU$^k$ networks.  

In this paper, we give a theoretical characterization of the eigenvalues and condition numbers of the mass and stiffness matrices $M$ and $K_s$ arising from ReLU$^k$ networks in terms of their mesh size $h$ and antipodally separation distance $\underline{h}$ (see \eqref{def_h_underlineh} below):
\begin{equation}
    \begin{split}
        &h^{-d}\lesssim\lambda_{\max}(M) \lesssim\underline{h}^{-d},\qquad \underline{h}^{2k+1}\lesssim\lambda_{\min}(M) \lesssim h^{2k+1},\\
        &h^{-d}\lesssim\lambda_{\max}(K_\ss) \lesssim\underline{h}^{-d},\qquad \underline{h}^{2(k-s)+1}\lesssim\lambda_{\min}(K_\ss) \lesssim h^{2(k-s)+1}.
    \end{split}
\end{equation}
In particular, when the parameters are antipodally quasi-uniform, i.e., $\underline{h}\simeq h$, we provide a sharp spectral analysis
\begin{equation}
    \begin{split}
        \lambda_j(M)\simeq nj^{-\frac{d+2k+1}{d}},\quad\lambda_j(K_\ss)\simeq nj^{-\frac{d+2(k-s)+1}{d}},\qquad j=1,\dots,n.
    \end{split}
\end{equation}
These asymptotics provide precise bounds for both the largest and smallest eigenvalues, and hence for $\kappa(M)$ and $\kappa(K_s)$, in terms of the number of neurons $n$, the activation order $k$, and the ambient dimension $d$.

The estimates reveal that the condition numbers decay at a polynomial rate matching that of the corresponding continuous kernel operators discussed above. To the best of our knowledge, they constitute the first complete theoretical analysis of the conditioning of such Gram matrices in shallow ReLU$^k$ networks. Moreover, the results quantify the intrinsic numerical stability of the least-squares formulation over $L_n^k$ and elucidate the trade-off between approximation power and conditioning dictated by the spectral structure of the underlying kernel.

The remainder of the paper is organized as follows. In Section~\ref{sec_prelim} we introduce the necessary background, notation, and preliminary results. Section~\ref{sec_mainresults} states our main theorems on the eigenvalue and condition number asymptotics. Detailed proofs are presented in Section~\ref{sec_proof_main}. Finally, Section~\ref{sec_conclu} summarizes our findings and discusses possible directions for future work.

\section{Preliminaries}\label{sec_prelim}

In this section, we establish the mathematical foundation for our analysis by introducing key notations and concepts related to functions on the unit sphere. The unit sphere in $\RR^{d+1}$ is defined as
$$\SS^d=\{x\in\RR^{d+1}:~|x|=1\},$$
which serves as the primary domain for the functions studied in this paper. We equip $\SS^d$ with its standard Riemannian metric inherited from $\RR^{d+1}$ and denote by $d\omega$ the normalized surface measure, where $d\omega = \frac{1}{\omega_d}d\sigma$ with $d\sigma$ being the standard surface measure and $\omega_d = \int_{\SS^d}d\sigma$ being the total surface area of $\SS^d$. This normalization ensures that $\int_{\SS^d}d\omega = 1$.

\subsection{Spherical harmonics}
Following the notation in \cite{liu2025achieving}, let $\YY_m$ denote the space of spherical harmonics of degree $m$ on $\SS^d$. For each $m\geq 0$, we fix an orthonormal basis
$$\{Y_{m,\ell}\}_{\ell=1}^{N(m)}$$
of $\YY_m$, where the dimension $N(m)$ is given by
\begin{equation}\label{eqn_def_Nm}
\begin{split}
    N(m) :&= \text{dim}(\YY_m) = \text{dim}(\mathbb{P}_m(\SS^d)) - \text{dim}(\mathbb{P}_{m-1}(\SS^d))\\
    &= \begin{cases}
        \displaystyle\frac{2m+d-1}{m}\binom{m+d-2}{d-1}, & \ m\neq0, \\
        1, & \ m=0.
    \end{cases}
\end{split}
\end{equation}

The space $\mathcal{L}^2(\SS^d)$ consists of square-integrable functions on $\SS^d$ with respect to the surface measure. Any function $f\in \mathcal{L}^2(\SS^d)$ admits a unique harmonic expansion
\begin{equation}
    f(\eta)=\sul_{m=0}^\infty\sul_{\ell=1}^{N(m)}\widehat{f}(m,\ell)Y_{m,\ell}(\eta),\qquad a.e. \quad \eta\in\SS^d,
\end{equation}
where the Fourier coefficients are given by the $\mathcal{L}^2$ inner product
\begin{equation*}
    \widehat{f}(m,\ell)=\left<f,Y_{m,\ell}\right>_{\mathcal{L}^2(\SS^d)}=\int_{\SS^d}f(\eta)Y_{m,\ell}(\eta)d\eta,\qquad m\in\NN,~1\leq\ell\leq N(m).
\end{equation*}

To quantify the distribution of a finite set of points on the sphere, we introduce the mesh size $h$ and antipodally separation distance $\underline{h}$, which describes how uniformly the points cover the sphere, and how distinctly the points are separated from each other. Formally, these distances are defined by
\begin{equation}\label{def_h_underlineh}
    h=\max_{\theta\in\SS^d}\min_{1\leq j\leq n}\rho(\theta,\theta_j^*),\quad \underline{h}=\min\big\{\min\limits_{i\neq j}\rho(\theta_i^*,\theta_j^*),\min\limits_{i\neq j}\rho(-\theta_i^*,\theta_j^*)\big\},
\end{equation}
where $\rho$ is the geodesic distance on $\SS^d$.

\begin{definition}[Antipodally  quasi-uniform]\label{def_anti_quasiunif}
    Let $d\in\NN$. A set of points $\{\theta_j^*\}_{j=1}^n\subset\SS^d$ is said to be antipodally quasi-uniform if
    \begin{equation}\label{eqn:quasi-uniform}
        h\lesssim\underline{h},
    \end{equation}
    where the implicit constant in \eqref{eqn:quasi-uniform} is independent of $n$.
\end{definition}

\subsection{Spectral Decomposition of Gram Matrices}

In order to accurately analyze the spectral properties of the mass matrix $M$ and the stiffness matrix $K$, we introduce their expansions using spherical harmonics and Legendre polynomials. Such decompositions allow us to explicitly characterize the eigenvalues of these matrices, thereby facilitating precise estimates of their condition numbers, as detailed in Section \ref{sec_proof_main}.

First, by orthogonality, we have
\begin{equation}\label{eqn:expan_sigma_square_G}
    \begin{split}
        &\int_{\SS^d}\sigma_k(\eta\cdot\omega)\sigma_k(\omega\cdot\theta)d\omega\\
        =&\int_{\SS^d}\left(\sul_{m=0}^\infty\widehat{\sigma_k}(m)Y_{m,\ell}(\eta)Y_{m,\ell}(\omega)\right)\left(\sul_{m=0}^\infty\sul_{\ell=1}^{N(m)}\widehat{\sigma_k}(m)Y_{m,\ell}(\omega)Y_{m,\ell}(\theta)\right)d\omega\\
        =&\sul_{m=0}^\infty\sul_{\ell=1}^{N(m)}\widehat{\sigma_k}(m)^2Y_{m,\ell}(\eta)Y_{m,\ell}(\theta)=\sul_{m=0}^\infty\widehat{\sigma_k}(m)^2p_m(\theta\cdot\eta),\qquad \theta,\eta\in\SS^d.
    \end{split}
\end{equation}
This directly yields a spectral decomposition of $M$ as
\begin{equation}\label{eqn:G_decomp0}
    M =\sul_{m=0}^\infty\widehat{\sigma_k}(m)^2P(m),
\end{equation}
where $P(m):=\left(p_m(\theta_i^*\cdot \theta_j^*)\right)_{i,j=1}^n$. the Legendre coefficients of $\sigma_k$ with respect to the polynomial basis $\{p_m\}_{m=0}^\infty$:
$$\widehat{\sigma_k}(m)=\frac{\int_{-1}^1p_m(t)\sigma_k(t)(1-t^2)^{\frac{d-2}{2}}dt}{\int_{-1}^1p_m(t)^2(1-t^2)^{\frac{d-2}{2}}dt}.$$
It was shown in \cite[Appendix D.2]{bach2017breaking} that for $m\geq k+1$ and $m\equiv k+1\mod{2}$,
\begin{equation}\label{eqn:hat_sigma_large}
    \widehat{\sigma_k}(m)=\frac{\omega_{d-1}}{\omega_d}\frac{k!(-1)^{(m-k-1)/2}}{2^m}\frac{\Gamma(d/2)\Gamma(m-k)}{\Gamma\left(\frac{m-k+1}{2}\right)\Gamma\left(\frac{m+d+k+1}{2}\right)},
\end{equation}
and 
\begin{equation}\label{eqn:hat_sigma_zero}
    \widehat{\sigma_k}(m)=0~\iff~m\geq k+1,~m\equiv k\mod{2}.
\end{equation}
In particular, the sequence $\{\widehat{\sigma_k}(m)\}_{m\geq k+1}$ vanishes for all even or odd integers greater than or equal to $k+1$.

We denote the function $\xi:~\{0,\dots,k\}\cup[k+1,\infty)$ by
\begin{equation}\label{eqn_def_xi_orig}
    \xi(t):=\left\{\begin{array}{ll}
        \widehat{\sigma_k}(t)^2, &\qquad t\in\{0,\dots,k\},  \\
        \displaystyle\lrt{\frac{\omega_{d-1}}{\omega_d}\frac{k!\Gamma(d/2)}{2^{k+1}\sqrt{\pi}}}^2 \lrt{\frac{\Gamma\lrt{\frac{t-k}{2}}}{\Gamma\lrt{\frac{t+d+k+1}{2}}}}^2, &\qquad t\in[k+1,\infty), 
    \end{array}\right.
\end{equation}
then \eqref{eqn:G_decomp0} can be written as
\begin{equation}\label{eqn:G_decomp}
    M =\sul_{m\in E}\xi(m)P(m),
\end{equation}
where
\begin{equation}\label{eqn_def_E}
    E:=\{n\in\NN:~\widehat{\sigma_k}(n)\neq0\}=\{0,\dots,k\}\cup\{2n+k+1:~n\in\NN\}.
\end{equation}

We show a similar expansion for stiffness matrices. It is easy to prove inductively that
\begin{equation*}
    \begin{split}
        &\int_{\SS^d}(\nabla_\omega^s\sigma_k(\eta\cdot\omega))\cdot(\nabla_\omega^s\sigma_k(\omega\cdot\theta))d\omega\\
        =&\int_{\SS^d}\big(\eta_1\nabla_\omega^{s-1}k\sigma_{k-1}(\eta\cdot\omega),\dots,\eta_d\nabla_\omega^{s-1}k\sigma_{k-1}(\eta\cdot\omega)\big)\cdot\big(\theta_1\nabla_\omega^{s-1}k\sigma_{k-1}(\theta\cdot\omega),\dots,\theta_d\nabla_\omega^{s-1}k\sigma_{k-1}(\theta\cdot\omega)\big)d\omega\\
        =&k^2(\theta\cdot\eta)\int_{\SS^d}(\nabla_\omega^{s-1}\sigma_{k-1}(\eta\cdot\omega))\cdot(\nabla_\omega^{s-1}\sigma_{k-1}(\omega\cdot\theta))d\omega\\
        =&\dots=(k)_s^2\int_{\SS^d}(\eta\cdot\theta)^s\sigma_{k-s}(\eta\cdot\omega)\sigma_{k-s}(\omega\cdot\theta)d\omega,
    \end{split}
\end{equation*}
where $(k)_s$ is the falling factorial $\displaystyle\frac{k!}{(k-s)!}$. Using the formula \eqref{eqn:expan_sigma_square_G}, we get
\begin{equation}\label{eqn:expan_sigma_square_K}
    \int_{\SS^d}(\nabla_\omega^s\sigma_k(\eta\cdot\omega))\cdot(\nabla_\omega\sigma_k(\omega\cdot\theta))d\omega=(k)_s^2\sul_{m=0}^\infty\widehat{\sigma_{k-s}}(m)^2(\eta\cdot\theta)^sp_m(\theta\cdot\eta)\qquad \theta,\eta\in\SS^d.
\end{equation}

Using the recurrence relation for normalized Legendre polynomials (see, e.g., \cite{szego1975orthogonal}), we have
\begin{equation*}
    (\theta\cdot\eta)p_m(\theta\cdot\eta)=\mathfrak{a}_mp_{m+1}(\theta\cdot\eta)+\mathfrak{b}_mp_{m-1}(\theta\cdot\eta),\qquad\theta,\eta\in\SS^d,
\end{equation*}
where
$$\mathfrak{a}_t=\frac{t+1}{2t+d+1},\quad\mathfrak{b}_t=\frac{t+d-1}{2t+d-3},\qquad t\geq0.$$

Denote the coefficients in terms of the recurrence relation by
\begin{equation*}
    \xi_0(t):=\left\{\begin{array}{ll}
        (k)_s^2\widehat{\sigma_{k-s}}(t)^2,&\qquad t=0,\dots,k-s,\\
    \displaystyle(k)_s^2\lrt{\frac{\omega_{d-1}}{\omega_d}\frac{k!\Gamma(d/2)}{2^{k+1}\sqrt{\pi}}}^2 \lrt{\frac{\Gamma\lrt{\frac{t-k+s}{2}}}{\Gamma\lrt{\frac{t+d+k-s+1}{2}}}}^2,&\qquad t\geq k-s+1,\\
    \end{array}\right.
\end{equation*}
and $\xi_s:~\{0,\dots,k\}\cup[k+1,\infty)$ be given iteratively by
\begin{equation}\label{eqn_def_xi}
    \xi_r(t)=\mathfrak{a}_{t-1}\xi_{r-1}(t-1)+\mathfrak{b}_{t+1}\xi_{r-1}(t+1),\qquad t\in \{0,\dots,k-s+r\}\cup[k-s+r+1,\infty)
\end{equation}
with the convention $\mathfrak{a}_{-1}=\mathfrak{a}_{-2}=\dots=0$.

Thus the stiffness matrix $K_s$ admits the following spectral decomposition:
\begin{equation}\label{eqn:K_decomp}
    K_s =\sul_{m\in E_s}\xi_s(m)P(m),
\end{equation}
where
\begin{equation}\label{eqn_def_Es}
    E_s:=\Big\{n\in\NN:~\int_{-1}^1(1-t^2)^{\frac{d-2}{2}}p_n(t)\sul_{m=0}^\infty\widehat{\sigma_{k-s}}(m)^2t^sp_m(t)dt\neq0\Big\}
\end{equation}

To facilitate subsequent spectral analysis, we follow the idea in \cite{liu2025achieving} and introduce the inverse of the $\beta$-th forward difference of the sequence $\{P(m)\}_{m\geq k+1}$ defined as
\begin{equation}\label{eqn:p_n^r_explicit}
    P_{\beta+1}(m)=\sul_{\nu=0}^mP_{\beta}(\nu)=(\delta_+^{-\beta}P)(m)=\sul_{\nu=0}^m\binom{m+\beta-\nu}{\beta}P(\nu),\qquad m\in\NN.
\end{equation}

\subsection{Auxiliary estimates and tools}

In this subsection, we collect several auxiliary results that will play a central role in the spectral analysis and condition number estimates of the Gram matrices in Section~\ref{sec_proof_main}. These results concern the asymptotic behavior of Legendre coefficients, the structure of difference-averaged matrices, and the construction of quadrature rules on the sphere.

We begin with a fundamental estimate on the decay of squared Legendre coefficients of the ReLU$^k$ activation, which exhibits discrete smoothness in the sense of forward differences.

\begin{lemma}[Lemma 3.1, \cite{liu2025achieving}]\label{lem:Bach}
Let $\xi$ be as in \eqref{eqn_def_xi_orig} and $E$ in \eqref{eqn_def_E}, then    \begin{equation}\label{eqn:def_theta}
        \begin{split}
        &E=\{0,\dots,k\}\cup\{m\in\NN:~m-k-1\in2\NN\},\\
            &(-1)^\beta\xi^{(\beta)}(t)\simeq t^{-(d+2k+1)-\beta},\qquad \beta=0,1,\dots.
        \end{split}
    \end{equation}
\end{lemma}

Next, we provide the corresponding estimate for the Legendre coefficient structure appearing in the stiffness matrix. It is obtained by combining the recurrence relation of Legendre polynomials with the decay properties of $\xi_0$ from Lemma~\ref{lem:Bach}.

\begin{lemma}\label{cor_mu_esti}
Let $\xi_s$ be as in \eqref{eqn_def_xi} and $E_s$ in \eqref{eqn_def_Es}, then for any $\beta\in\NN$,
\begin{equation}\label{eqn:def_lambda}
    \begin{split}
     &E_s=\{0,\dots,k\}\cup\{m\in\NN:~m-k-1\in2\NN\}=E,\\
        &(-1)^\beta\xi_s^{(\beta)}(t)\simeq t^{-(d+2k+1-2s)-\beta},\qquad t\geq k+1.
    \end{split}
\end{equation}
\end{lemma}

\begin{proof}
Denote
$$F_{r}:=\{m\in\NN:~\xi_r(m)\neq0\},~r=0,\dots,s.$$
By Lemma \ref{lem:Bach},
$$F_{0}=\{0,\dots,k-s\}\cup\{m\in\NN:~m-(k-s+1)\in2\NN\}.$$
We apply induction, the induction hypothesis is
$$F_{r-1}=\{0,\dots,k-s+r-1\}\cup\{m\in\NN:~m-(k-s+1+r-1)\in2\NN\}$$
and we consider $F_r$. For $m\leq k-s+r$, we have
$$m-1\in\{0,\dots,k-s+r-1\}\subset F_{r-1}.$$
Then
$$\xi_r(m)=\mathfrak{a}_{m-1}\xi_{r-1}(m-1)+\mathfrak{b}_{m+1}\xi_{r-1}(m+1)\geq\mathfrak{a}_{m-1}\xi_{r-1}(m-1)>0.$$
Next, for $m-(k-s+1-r)\in2\NN$, then $$m+1\in\{m\in\NN:~m-(k-s+1+r-1)\in2\NN\},$$
as a result,
$$\xi_r(m)=\mathfrak{a}_{m-1}\xi_{r-1}(m-1)+\mathfrak{b}_{m+1}\xi_{r-1}(m+1)\geq\mathfrak{b}_{m+1}\xi_{r-1}(m+1)>0.$$
For $m-(k-s+1-r)\in2\NN+1$, we have
\begin{equation*}
    \begin{split}
        &m-1-(k-s+1+r-1)=m-(k-s+1-r)\in2\NN+1,\\
        &m+1-(k-s+1+r-1)=m+2-(k-s+1-r)\in2\NN+1,
    \end{split}
\end{equation*}
which implies $m-1,m+1\notin F_{r-1}$. Therefore,
$$\xi_r(m)=\mathfrak{a}_{m-1}\xi_{r-1}(m-1)+\mathfrak{b}_{m+1}\xi_{r-1}(m+1)=0.$$
Combining all the cases together, we have
$$F_{r}=\{0,\dots,k-s+1+r\}\cup\{m\in\NN:~m-(k-s+1+r)\in2\NN\}$$
and consequently
$$E_s=F_{s}=\{0,\dots,k\}\cup\{m\in\NN:~m-k-1\in2\NN\}.$$

Now we show the second estimation by induction over $r$. For $r=0$ this is just Lemma \ref{lem:Bach}. Suppose for $r-1$, we have
\begin{equation}\label{eqn_ind_hyp_Bach}
    \xi_{r-1}^{(\beta)}(t)\simeq t^{-(d+2k+1-2s)-\beta},\qquad t\geq k-s+r.
\end{equation}
Then
$$
\xi_r(t)=\frac{t}{2t+d-1}\xi_{r-1}(t-1)+\frac{t+d}{2t+d-1}\xi_{r-1}(t+1),\qquad t\geq k-s+r+1.
$$
Applying Leibniz rule yields
\begin{equation*}
    \xi_{r}^{(\beta)}(t)=\sul_{j=0}^\beta\binom{\beta}{j}\Big[\xi_{r-1}^{(\beta-j)}(t-1)\Big(\frac{d}{dt}\Big)^j\Big(\frac{t}{2t+d-1}\Big)+\xi_{r-1}^{(\beta-j)}(t+1)\Big(\frac{d}{dt}\Big)^j\Big(\frac{t+d}{2t+d-1}\Big)\Big].
\end{equation*}
Using the induction hypothesis \eqref{eqn_ind_hyp_Bach}
\begin{equation}
    \xi_{r-1}^{(\beta-j)}(t-1)\simeq \xi_{r-1}^{(\beta-j)}(t+1)\simeq (-1)^{\beta-j} t^{-(d+2k+1-2s)-(\beta-j)},
\end{equation}
and the fact (verifiable by direct computation)
\begin{equation}
    \Big(\frac{d}{dt}\Big)^j\Big(\frac{t}{2t+d-1}\Big)\simeq (-1)^{j} t^{-j},\quad 
    \Big(\frac{d}{dt}\Big)^j\Big(\frac{t+d}{2t+d-1}\Big)\simeq (-1)^{j} t^{-j},
\end{equation}
we obtain
\begin{equation*}
    \xi_{r}^{(\beta)}(t)\simeq \sul_{j=0}^\beta\binom{\beta}{j}
    \Big[
    (-1)^{\beta-j} t^{-(d+2k+1-2s)-(\beta-j)} \cdot (-1)^j t^{-j}
    \Big]
    \simeq (-1)^\beta t^{-(d+2k+1-2s)-\beta}.
\end{equation*}
\end{proof}

% Now we turn to the algebraic structure of the matrices $P_\beta(m)$ arising from successive summations of the Legendre Gram matrices $P(m)$. These matrices play a crucial role in capturing the localization properties of the Gram matrices and in estimating extrema eigenvalues.

% \begin{lemma}[Lemma 3.4, \cite{liu2025achieving}]\label{lem:property_p_nr}
% Let $r\in\NN$. The matrices $\left\{P_\beta(m)\right\}$ satisfy the following properties:
% \begin{itemize}
%     \item[(a)] Each $P_\beta(m)$ is semi-positive definite;
    
%     \item[(b)] The diagonal elements of $P_\beta(m)$ are equal and satisfy
%     \begin{equation}\label{eqn:diag_p_nr}
%         \lrt{P_\beta(m)}_{i,i}=\max\limits_{1\leq i,j\leq n}\left|\lrt{P_\beta(m)}_{i,j}\right|\simeq m^{\beta+d-1};
%     \end{equation}
    
%     \item[(c)] Let $\alpha=\lceil\frac{d+2}{2}\rceil$, then

%     \begin{equation}
%         \sul_{j\neq i}|\lrt{P_{\alpha}(m)}_{i,j} | 
%         \lesssim m^{\frac{d-1}{2}}\underline{h}^{-(\frac{d-1}{2}+\alpha)}+m^{d-1}
%     \end{equation}
%     \begin{equation}\label{eqn:est_Pns_2norm}
%         \left\|P_\alpha(m)\right\|_2\lesssim m^{\alpha+d-1}\lrt{\frac{h}{\underline{h}}}^{\frac{d-1}{2}+\alpha}+m^{d-1},
%     \end{equation}
%     where $h$ and $\underline{h}$ are as defined in \eqref{eqn:quasi-uniform}.
% \end{itemize}
% \end{lemma}

\section{Main Results}\label{sec_mainresults}

This section presents the main theoretical contributions of the paper: sharp estimates for the condition numbers of the mass matrix $M$ and the stiffness matrix $K$ associated with ReLU$^k$ neural networks on the unit sphere. These results characterize how the condition numbers depend on the number of parameters $n$, the activation order $k$, and the dimension $d$.

Throughout this paper, we order the eigenvalues of the matrices in the descending order:
\begin{equation}
    \begin{split}
        &\lambda_{\max}(M)=\lambda_1(M)\geq\lambda_2(M)\geq\dots\geq\lambda_n(M)=\lambda_{\min}(M),\\
        &\lambda_{\max}(K_\ss)=\lambda_1(K_\ss)\geq\lambda_2(K_\ss)\geq\dots\geq\lambda_n(K_\ss)=\lambda_{\min}(K_\ss).
    \end{split}
\end{equation}
We first state the spectral bounds for the mass and stiffness matrices.
%[Condition number of mass and stiffness matrices]
\begin{theorem}\label{thm:eigenvalue}
Let $d,k\in\NN$. For any $n\in\NN$ and $\{\theta_j^*\}_{j=1}^n\subset\SS^d$, the mass matrix
\[
M=\left(\int_{\SS^d}\sigma_k(\theta_i^*\cdot\eta)\sigma_k(\theta_j^*\cdot\eta)\,d\eta\right)_{i,j=1}^n
\]
and stiffness matrix, $s\leq k$,
\[
K_s=\left(\int_{\SS^d}(\nabla_\omega^s\sigma_k(\theta_i^*\cdot\eta))\cdot(\nabla_\omega^s\sigma_k(\theta_j^*\cdot\eta))\,d\eta\right)_{i,j=1}^n
\]
satisfy the bounds
\begin{equation}
\begin{split}
    &h^{-d}\lesssim\lambda_{1}(M)\lesssim \underline{h}^{-d},\qquad\underline{h}^{2k+1}\lesssim\lambda_{n}(M)\lesssim  h^{2k+1},
\end{split}
\end{equation}
and
\begin{equation}
\begin{split}
    h^{-d}\lesssim\lambda_{1}(K_\ss)\lesssim \underline{h}^{-d},\qquad \underline{h}^{2(k-s)+1}\lesssim\lambda_{n}(K_\ss)\lesssim  h^{2(k-s)+1},
\end{split}
\end{equation}
respectively. In particular, under the assumption $h\simeq \underline{h}$, the condition number satisfies
\begin{equation}
    \kappa(M)\simeq n^{\frac{d+2k+1}{d}},\quad \kappa(K_\ss)\simeq n^{\frac{d+2(k-s)+1}{d}}
\end{equation}
\end{theorem}

\begin{remark}
It is instructive to compare the scaling behavior of condition numbers and approximation rates in linearized ReLU$^k$ networks with that of classical finite element methods (FEM). In FEM with $k$-th order polynomial basis on a $d$-dimensional domain, the $L^2$ approximation error decays as $n^{-(k+1)/d}$. The mass matrix is well conditioned with $\kappa(M_{\mathrm{FEM}})=\mathcal O(1)$, essentially independent of $n$, while the stiffness matrix becomes increasingly ill-conditioned as higher derivatives are involved, with $\kappa(K_{\mathrm{FEM}})\sim n^{2s/d}$ for derivative order $s$ on quasi-uniform meshes.

In contrast, the ReLU$^k$ network achieves a faster approximation rate of $n^{-(d+2k+1)/(2d)}$, but our results show that this comes at the cost of a significantly larger condition number in the quasi-uniform setting:
\[
\kappa(M) \sim n^{\frac{d+2k+1}{d}}, \qquad \kappa(K_s) \sim n^{\frac{d+2(k-s)+1}{d}}.
\]
Interestingly, while FEM stiffness matrices deteriorate as $s$ increases, in the neural network case the opposite occurs: taking derivatives effectively reduces the power of the ReLU activation (e.g.\ $\partial_x \mathrm{ReLU}^k \sim \mathrm{ReLU}^{k-1}$), simplifying the basis and improving stability. Thus, in both frameworks a trade-off between approximation power and numerical stability is present, but manifested in opposite ways. In any case, this resonates with a recurring theme in approximation theory and machine learning alike: there is no free lunch.
\end{remark}

The quadrature formula on the sphere, which represents integrals by weighted sums over scattered points, is a key tool in this paper. The following lemma allows us to construct the discrete version of the orthonormal basis $Y_{m,\ell}$, which can be used to analyze the spectral of $M$ and $K_\ss$ further.

\begin{lemma}[{\cite[Theorem 4.1]{mhaskar2001spherical}, \cite[Theorem 3.2]{narcowich2006sobolev}}]\label{lem:quadrature}
There exist nonnegative weights $\tau_1,\dots,\tau_n$ with $\tau_j\lesssim h^d$, and a constant $C_1$ independent of $n$ and $h$, such that
\begin{equation}\label{eqn:quadrature_formula}
    \int_{\SS^d}p(\eta)q(\eta)d\eta=\sul_{j=1}^n\tau_jp(\theta_j^*)q(\theta_j^*),\quad \forall p,q\in\mathbb{P}_J(\SS^d),
\end{equation}
where $J=\lfloor C_1h^{-1}\rfloor$. Moreover, if the separation distance satisfies
\begin{equation}
    \mil_{i\neq j}\rho(\theta_i^*,\theta_j^*)\gtrsim h,
\end{equation}
then
\begin{equation}
    \tau_j\simeq h^d\simeq n^{-1},\qquad j=1,\dots,n.
\end{equation}
All of the corresponding constants are independent of $n$ and $\{\theta_j^*\}_{j=1}^n$.
\end{lemma}

\begin{theorem}\label{thm_spectr_main}
    Let $d,k,n$, $\{\theta_j^*\}_{j=1}^n$ be as in Theorem \ref{thm:eigenvalue}, $J$ and $\{\tau_j\}_{j=1}^n$ be as in Lemma \ref{lem:quadrature}.
    \begin{enumerate}
        \item [(1)] If $h\simeq\underline{h}$,
        then
        \begin{equation}\label{eqn_spectr_main1}
        \lambda_j(M)\simeq nj^{-\frac{d+2k+1}{d}},\qquad j=1,\dots,n.
    \end{equation}
    and
    \begin{equation}\label{eqn_spectr_main1_stiff}
        \begin{split}
            \lambda_j(K_\ss)\simeq nj^{-\frac{d+2(k-s)+1}{d}},\qquad j=1,\dots,n.
        \end{split}
    \end{equation}
        \item [(2)] Moreover, if $\tau_1=\dots=\tau_n=n^{-1}$, then there exists $C_2$ and $J_0$ independent of $n$, such that for $m$ satisfying:
    \begin{eqnarray*}
        &&m+2\in E\cap\{J_0,J_0+1,\dots,J\},\\
        &&\xi(m+2)\geq C_2h^{2k+1},\\
        &&\xi_s(m+2)\geq C_2h^{2(k-s)+1},
    \end{eqnarray*}
    we have
    \begin{equation}\label{eqn_spectr_main2}
        \begin{split}
            &n\big(\xi(m+2)+C_2h^{d+2k+1}\big)\geq\lambda_{d_{m}+1}(M)\geq\dots\geq \lambda_{d_{m+2}}(M)\geq n\xi(m+2).\\
            &n\big(\xi_\ss(m+2)+C_2h^{d+2(k-s)+1}\big)\geq\lambda_{d_{m}+1}(K_\ss)\geq\dots\geq\lambda_{d_{m+2}}(K_\ss)\geq n\xi_\ss(m+2),
        \end{split}
    \end{equation}
    where
    \begin{equation}\label{eqn_def_dm}
        d_{m}=\sul_{i\in E,~i\leq m}N(i).
    \end{equation}
    \end{enumerate}
\end{theorem}
\begin{remark}
The stronger hypotheses $$\tau_1=\dots=\tau_n=n^{-1}$$ in Theorem~\ref{thm_spectr_main} is natural and meaningful. In fact, the existence of spherical $t$-designs with exactly equal weights $n^{-1}$ has been proved in \cite{bondarenko2013optimal}.
\end{remark}

From a practical standpoint, these results establish quantitative control over the condition numbers of the Gram and stiffness matrices associated with shallow ReLU$^k$ networks. This directly impacts the well-posedness and numerical solvability of the least-squares system governing the optimization of neural network models. The role of $M$ and $K_s$ here is closely analogous to the mass and stiffness matrices in FEM, but with a crucial distinction: in the neural network setting, the spectral structure leads to larger condition numbers overall but more favorable dependence on the derivative order $s$, which make it suitable for solving high order PDEs. Our estimates thus provide a theoretical foundation for understanding the stability and generalization behavior of neural networks through the lens of spectral conditioning, while also clarifying their differences from classical finite element discretizations.

% \section{Applications to effective dimension and generalization}\label{sec:applications}

    The effective degrees of freedom $d_{\mathrm{eff}}(\lambda)$ provide a crucial measure of the capacity of the model to fit the data. In particular, they help to quantify the trade-off between model complexity and generalization performance, which is a central theme in statistical learning theory. Our spectral estimates enable precise characterization of $d_{\mathrm{eff}}(\lambda)$ for the mass matrix $M$ associated with ReLU$^k$ networks, shedding light on how the choice of regularization parameter $\lambda$ influences the model's effective complexity.

% \paragraph{Definition and interpretation.}

% \begin{proposition}[Variance bound via $d_{\mathrm{eff}}(M,\lambda)$]\label{prop:var-deff}
% Under the setup of the remark, for any fixed design $\{\omega_i\}_{i=1}^m$ and any $\lambda>0$,
% \[
% \mathbb E\!\left[\|\hat f-f^\star\|_{\mathcal L^2(\SS^d)}^2\,\middle|\,\{\omega_i\}\right]
% \;\le\;\frac{\sigma^2}{m}\operatorname{tr}\!\big(M\,(G+\lambda I)^{-1}\big).
% \]
% If, in addition, $G$ is close to $M$ in operator norm (e.g.\ $m$ large and $\omega_i\sim\mathrm{Unif}(\SS^d)$), then
% \[
% \mathbb E\!\left[\|\hat f-f^\star\|_{\mathcal L^2(\SS^d)}^2\right]
% \;\lesssim\;\frac{\sigma^2}{m}\,d_{\mathrm{eff}}(M,\lambda).
% \]
% \end{proposition}

% \begin{proof}
% Write $\hat a-a^\star=(G+\lambda I)^{-1}\frac{1}{m}\Phi^\top\varepsilon - \lambda(G+\lambda I)^{-1}a^\star$. The variance term is
% \[
% \frac{1}{m^2}\mathbb E\!\big[\varepsilon^\top \Phi(G+\lambda I)^{-1}M(G+\lambda I)^{-1}\Phi^\top\varepsilon\big]
% =\frac{\sigma^2}{m}\operatorname{tr}\!\big((G+\lambda I)^{-1}M(G+\lambda I)^{-1}G\big).
% \]
% Since $0\preceq G\preceq G+\lambda I$, we have $(G+\lambda I)^{-1}G(G+\lambda I)^{-1}\preceq (G+\lambda I)^{-1}$. Therefore the variance is bounded by $\frac{\sigma^2}{m}\operatorname{tr}\!\big(M(G+\lambda I)^{-1}\big)$. Concentration of $G$ around $M$ (matched distribution) then yields the claimed approximation by $d_{\mathrm{eff}}(M,\lambda)$.
% \end{proof}

For a positive regularization parameter $\lambda$, the \emph{effective degrees of freedom} (or \emph{effective dimension})
\begin{equation}\label{eq:def-deff}
d_{\mathrm{eff}}(\lambda):=\mathrm{tr}\big(M(M+\lambda I)^{-1}\big)=\sum_{j=1}^n \frac{\lambda_j(M)}{\lambda_j(M)+\lambda}
\end{equation}
measures the number of spectral directions that are not suppressed by Tikhonov regularization. It satisfies $0\le d_{\mathrm{eff}}(\lambda)\le n$, it is monotonically decreasing in $\lambda$, and obeys the endpoint bounds
$d_{\mathrm{eff}}(0^+)=\mathrm{rank}(M)$ and $d_{\mathrm{eff}}(\lambda)\le \lambda^{-1}\mathrm{tr}(M)$.
In kernel ridge regression, $d_{\mathrm{eff}}(\lambda)$ governs the variance term of the risk (see, e.g., \cite{caponnetto2007optimal,rudi2017generalization}).

Under antipodally quasi-uniform parameters ($h\simeq \underline{h}$), Theorem~\ref{thm_spectr_main}(1) yields
\begin{equation}\label{eq:spectral-law}
\lambda_j(M)\simeq n\, j^{-\alpha},\qquad \alpha:=\frac{d+2k+1}{d}>1,\qquad 1\le j\le n.
\end{equation}
Consequently, for thresholds $t$ in the interior spectral range $t\in [c\,n^{1-\alpha}, C\,n]$,
the eigenvalue counting function $N_M(t):=\#\{j:\lambda_j(M)\ge t\}$ obeys the Weyl-type estimate
\begin{equation}\label{eq:counting}
N_M(t)\ \simeq\ \Big(\frac{n}{t}\Big)^{1/\alpha}.
\end{equation}

\begin{corollary}[Effective dimension ]\label{cor:effdim}
Assume $h\simeq \underline{h}$ so that \eqref{eq:spectral-law} holds. Let $\alpha=(d+2k+1)/d$.
There exist constants $c_1,c_2,c_3,c_4>0$, independent of $n$ and $\lambda$, such that for all
\[
c_1\,n^{1-\alpha}\ \le\ \lambda\ \le\ c_2\,n,
\]
the effective degrees of freedom satisfy
\begin{equation}\label{eq:deff-scaling}
c_3\Big(\frac{n}{\lambda}\Big)^{1/\alpha}\ \lesssim\ d_{\mathrm{eff}}(\lambda)\ \lesssim\ c_4\Big(\frac{n}{\lambda}\Big)^{1/\alpha}.
\end{equation}
In particular,
\[
d_{\mathrm{eff}}(\lambda)\ \simeq\ \Big(\frac{n}{\lambda}\Big)^{\frac{d}{\,d+2k+1\,}}.
\]
Moreover, for all $\lambda>0$ one has the global sandwich bounds
\begin{equation}\label{eq:deff-global}
d_{\mathrm{eff}}(\lambda)\ \le\ C\,\min\Big\{\,n,\ \Big(\frac{n}{\lambda}\Big)^{1/\alpha}+ \frac{n}{\lambda}\,\Big\},
\qquad
d_{\mathrm{eff}}(\lambda)\ \ge\ c\,\min\Big\{\,n,\ \Big(\frac{n}{\lambda}\Big)^{1/\alpha}\,\Big\},
\end{equation}
where $c,C>0$ depend only on $d,k$.
\end{corollary}

\begin{proof}
By Theorem~\ref{thm_spectr_main}(1) there are $c_-,c_+>0$ such that
$c_-\,n\, j^{-\alpha}\le \lambda_j(M)\le c_+\,n\, j^{-\alpha}$ for $1\le j\le n$.
For the lower bound in \eqref{eq:deff-scaling}, set
$J_-:=\big\lfloor (c_- n/(2\lambda))^{1/\alpha}\big\rfloor$. Then for $j\le J_-$,
$\lambda_j(M)\ge 2\lambda$, hence $\lambda_j(M)/(\lambda_j(M)+\lambda)\ge 2/3$ and
\[
d_{\mathrm{eff}}(\lambda)\ =\ \sum_{j=1}^{n}\frac{\lambda_j(M)}{\lambda_j(M)+\lambda}
\ \ge\ \sum_{j=1}^{J_-}\frac{\lambda_j(M)}{\lambda_j(M)+\lambda}\ \gtrsim\ J_-\ \gtrsim\ (n/\lambda)^{1/\alpha}.
\]
For the upper bound, define $J_+:=\big\lceil (c_+ n/\lambda)^{1/\alpha}\big\rceil$ and split
\[
d_{\mathrm{eff}}(\lambda)\ \le\ J_+\ +\ \frac{1}{\lambda}\sum_{j>J_+}\lambda_j(M)
\ \le\ J_+\ +\ \frac{c_+ n}{\lambda}\sum_{j>J_+}j^{-\alpha}.
\]
Since $\alpha>1$, $\sum_{j>J_+} j^{-\alpha}\lesssim \int_{J_+}^{\infty}x^{-\alpha}\,dx\simeq J_+^{-(\alpha-1)}$, hence
\[
d_{\mathrm{eff}}(\lambda)\ \lesssim\ J_+\ +\ \frac{c_+ n}{\lambda}\,J_+^{-(\alpha-1)}
\ \simeq\ (n/\lambda)^{1/\alpha}.
\]
This proves \eqref{eq:deff-scaling} on the stated $\lambda$-range. The global bounds \eqref{eq:deff-global} follow from the same splitting, together with
$d_{\mathrm{eff}}(\lambda)\le n$ and $d_{\mathrm{eff}}(\lambda)\le \lambda^{-1}\mathrm{tr}(M)\lesssim n/\lambda$
(since $\sum_{j=1}^n j^{-\alpha}$ is uniformly bounded for $\alpha>1$).
\end{proof}

The rest of the paper is devoted to the proof of Theorems~\ref{thm:eigenvalue} and~\ref{thm_spectr_main}, which is carried out in Section~\ref{sec_proof_main}.

% We end this section by introducing a quadrature formula on the sphere, which represents integrals by weighted sums over scattered points. The following result, a consequence of the Marcinkiewicz-Zygmund inequality for quasi-uniform points, shows the existence of a stable quadrature rule with well-behaved weights.

% \begin{lemma}\label{lem:quadrature}
% Let $\{\theta_j^*\}_{j=1}^n \subset \SS^d$ be a quasi-uniform set of points with mesh norm $h$ and separation distance $\underline{h}$ satisfying $h \simeq \underline{h}$. Then there exist positive weights $\{\tau_j\}_{j=1}^n$ and constants $C_1, C_2, C_3 > 0$ independent of $n$, such that
% \begin{enumerate}
%     \item[(a)] The weights are uniformly bounded above and below:
%     $$ C_1 h^d \le \tau_j \le C_2 h^d, \qquad j=1,\dots,n. $$
%     \item[(b)] The weights provide a quadrature rule for all spherical polynomials $P \in \mathbb{P}_J(\SS^d)$ of degree at most $J = \lfloor C_3 h^{-1} \rfloor$:
%     \begin{equation}\label{eqn:quadrature_formula}
%         \int_{\SS^d} P(\eta) d\eta = \sum_{j=1}^n \tau_j P(\theta_j^*).
%     \end{equation}
% \end{enumerate}
% \end{lemma}

\section{Proofs of the Main Results}\label{sec_proof_main}
% In order to prove Theorem \ref{thm:eigenvalue}, we fill out these vanishing terms in \eqref{eqn:G_decomp} and \eqref{eqn:K_decomp} to denote
% \begin{equation}\label{eqn_def_tMK}
%     \tM:=\sul_{m=0}^\infty\xi_0(m)P(m),\quad\tK:=\sul_{m=0}^\infty\xi_s(m)P(m)
% \end{equation}
% where $\xi_0$, $\xi_s$ are the functions defined in \eqref{eqn_def_xi}.

In the theory of harmonic analysis, we usually divide the summations of type \eqref{eqn:G_decomp} and \eqref{eqn:K_decomp} into ``needlets" by a smooth function $\zeta$ satisfying the following conditions (see, e.g., \cite[(3.6)]{petrushev2005localized})
\begin{eqnarray*}
&&\zeta\in\mathcal{C}^\infty(\RR),\quad\zeta\geq0,\quad\mathrm{supp}(\zeta)\subset[1/2,2],\\
&&\zeta(t)>c_1>0,\qquad t\in[3/5,5/3],\\
&&\zeta(t)+\zeta(2t)=1,\quad t\in[1/2,1].
\end{eqnarray*}
A typical property of $\zeta$ is the partition of unity:
\begin{equation}\label{eqn_zeta_sum1}
    1=\sul_{q=0}^\infty\zeta(2^{-q}m),\qquad m\in\NN_+.
\end{equation}

Now we can write \eqref{eqn:G_decomp} and \eqref{eqn:K_decomp} as
\begin{equation}
    M=\sul_{q=0}^\infty Q(q),\quad K_s=\sul_{q=0}^\infty Q_s(q),
\end{equation}
where
\begin{equation}\label{eqn_sob_norm_s_r}
    \begin{split}
        &Q(0)=\xi(0)P(0)+\xi(1)P(1),\quad Q(q)=\sul_{m\in E}\zeta\big(\frac{m}{2^q}\big)\xi(m)P(m),\qquad q=1,2,\dots,\\
        &Q_s(0)=\xi_s(0)P(0)+\xi_s(1)P(1),\quad Q_s(q)=\sul_{m\in E_s}\zeta\big(\frac{m}{2^q}\big)\xi_s(m)P(m),\qquad q=1,2,\dots.
    \end{split}
\end{equation}
For $q\geq\lceil\log_2(k+1)\rceil+1$, we define some related matrices with better properties:
\begin{equation}
    R(q)=\binom{\tQ_+(q)\quad\tQ_-(q)}{\tQ_-(q)\quad\tQ_+(q)},\qquad R_s(q)=\binom{\tQ_{s,+}(q)\quad\tQ_{s,-}(q)}{\tQ_{s,-}(q)\quad\tQ_{s,+}(q)},
\end{equation}
where
\begin{equation}
    \begin{split}
        &\tQ_+(q)=\sul_{m=2^{q-1}+1}^{2^{q+1}}\zeta\big(\frac{m}{2^q}\big)\xi(m)P(m),\quad \tQ_-(q)=\sul_{m=2^{q-1}+1}^{2^{q+1}}(-1)^{m-k-1}\zeta\big(\frac{m}{2^q}\big)\xi(m)P(m),\\
        &\tQ_{s,+}(q)=\sul_{m=2^{q-1}+1}^{2^{q+1}}\zeta\big(\frac{m}{2^q}\big)\xi_s(m)P(m),\quad \tQ_{s,-}(q)=\sul_{m=2^{q-1}+1}^{2^{q+1}}(-1)^{m-k-1}\zeta\big(\frac{m}{2^q}\big)\xi_s(m)P(m).
    \end{split}
\end{equation}
It has been proved in Lemma \ref{lem:Bach} and \ref{cor_mu_esti} that for $m\geq k+1$, $m=0$ if and only if $m\equiv k+1\mod 2$. Thus the following formula holds for any $a\in\RR^n$:
\begin{equation}\label{eqn_rel_R_Q}
    \begin{split}
        &(a^\top,a^\top)R(q)\binom{a}{a}=(a^\top,a^\top)\begin{pmatrix}
\tQ_+(q) & \tQ_-(q) \\
\tQ_-(q) & \tQ_+(q)
\end{pmatrix}\binom{a}{a}=2a^\top\tQ_+(q)a+2a^\top\tQ_-(q)a\\
        =&2a^\top\Big[\sul_{m=2^{q-1}+1}^{2^{q+1}}(1+(-1)^{m-k-1})\zeta\big(\frac{m}{2^q}\big)\xi(m)P(m)\Big]a
        =4a^\top\Big[\sul_{m\in E}\zeta\big(\frac{m}{2^q}\big)\xi(m)P(m)\Big]a\\
        =&4a^\top Q(q)a.
    \end{split}
\end{equation}
Similarly,
\begin{equation}\label{eqn_rel_Rs_Qs}
    (a^\top,a^\top)R_s(q)\binom{a}{a}=4a^\top Q_s(q)a.
\end{equation}
In the rest of this section, we concentrate on the matrices $R(q)$ and $R_s(q)$.

\subsection{Estimating the norms of needlet matrices}
It is known that $p_m(t)=(-1)^{m}p_m(-t)$ for all $m\in\NN$, then
\begin{equation*}
    \begin{split}
        &\tQ_-(q)_{i,j}=\sul_{m=2^{q-1}+1}^{2^{q+1}}(-1)^{m-k-1}\zeta\big(\frac{m}{2^q}\big)\xi(m)p_m(\theta_i^*\cdot\theta_j^*)=(-1)^{k+1}\sul_{m=2^{q-1}+1}^{2^{q+1}}\zeta\big(\frac{m}{2^q}\big)\xi(m)p_m(-\theta_i^*\cdot\theta_j^*),\\
        &\tQ_{s,-}(q)_{i,j}=\sul_{m=2^{q-1}+1}^{2^{q+1}}(-1)^{m-k-1}\zeta\big(\frac{m}{2^q}\big)\xi_s(m)p_m(\theta_i^*\cdot\theta_j^*)=(-1)^{k+1}\sul_{m=2^{q-1}+1}^{2^{q+1}}\zeta\big(\frac{m}{2^q}\big)\xi_s(m)p_m(-\theta_i^*\cdot\theta_j^*).
    \end{split}
\end{equation*}

Denote polynomials
\begin{equation*}
    \begin{split}
        &\varphi_{q}(t)=\sul_{m=2^{q-1}+1}^{2^{q+1}}\zeta\big(\frac{m}{2^q}\big)\xi(m)p_m(t),\qquad t\in[-1,1],\\
        &\varphi_{s,q}(t)=\sul_{m=2^{q-1}+1}^{2^{q+1}}\zeta\big(\frac{m}{2^q}\big)\xi_s(m)p_m(t),\qquad t\in[-1,1],
    \end{split}
\end{equation*}
then
\begin{equation}
    R(q)_{i,j}=\varphi_q(\eta_i^*\cdot\eta_j^*),\quad R_s(q)_{i,j}=\varphi_{s,q}(\eta_i^*\cdot\eta_j^*),\qquad i,j=1,\dots,2n,
\end{equation}
where
\begin{equation*}
    (\eta_1^*,\dots,\eta_{2n}^*)=(\theta_1^*,\dots,\theta_n^*,-\theta_1^*,\dots,-\theta_n^*).
\end{equation*}
Then \eqref{def_h_underlineh} gives
\begin{equation}
    \min\limits_{i\neq j}\rho(\eta_i^*,\eta_j^*)=\underline{h}.
\end{equation}

The localization property of polynomials $\{\varphi_{q}\}_{q=1}^\infty$ and $\{\varphi_{\ss,q}\}_{q=1}^\infty$ can be regarded as an immediate consequence of the following standard localization property in orthogonal polynomial theory (see, e.g., \cite{petrushev2005localized,ivanov2010sub,filbir2009filter, mhaskar2010eignets}).
\begin{theorem}[\cite{petrushev2005localized}, Theorem 2.4]\label{thm_petru_xu}
    Let $a\in\mathcal{C}^\alpha[0,\infty)$ and $\mathrm{supp}(a)\subset[1/2,2]$, $\{P_m\}_{m=0}^\infty$ be the standard Jacobi polynomials with index $\frac{d-2}{2}$, then for any $0\leq\theta,\phi\leq\pi$, the function
    \begin{equation}
        l_n(\cos\theta,\cos\phi)=\sul_{m=0}^\infty a\big(\frac{m}{n}\big)\frac{P_m(\cos\theta)P_m(\cos\phi)}{\|P_m\|_{w_d}^2}
    \end{equation}
    has the bound
    \begin{equation*}
        |l_n(\cos\theta,\cos\phi)|\lesssim\frac{n\mal_{1\leq\beta\leq\alpha}\|a^{(\alpha)}\|_{\mathcal{L}^1}}{\big(\sqrt{\sin^2\theta+n^{-2}}\big)^{\frac{d-1}{2}}\big(\sqrt{\sin^2\phi+n^{-2}}\big)^{\frac{d-1}{2}}(1+n|\theta-\phi|)^{\alpha-(2d-1)}}.
    \end{equation*}
    
\end{theorem}

\begin{lemma}\label{lem_local_poly}
For any $\gamma>0$,
        \begin{equation}
            \begin{split}
                &|\varphi_{q}(t)|\lesssim\frac{2^{-q(d+2k)}}{(\sqrt{1+t}+2^{-q})^{\frac{d-1}{2}}(\sqrt{1-t}+2^{-q})^{\frac{d-1}{2}}(1+2^q\sqrt{1-t})^{\gamma}},\\
                &|\varphi_{\ss,q}(t)|\lesssim\frac{2^{-q(d+2k-2\ss)}}{(\sqrt{1+t}+2^{-q})^{\frac{d-1}{2}}(\sqrt{1-t}+2^{-q})^{\frac{d-1}{2}}(1+2^q\sqrt{1-t})^{\gamma}},
            \end{split}
        \end{equation}
where the corresponding constants are independent of $q$ and $t$.
\end{lemma}

\begin{proof}
    We can write $p_m(1)=\|p_m\|_{w_d}^2$ and rewrite $\varphi_{\ss,q}$ as
    \begin{equation}
        \varphi_{\ss,q}(t)=\sul_{m=2^{q-1}+1}^{2^{q+1}}\zeta\big(\frac{m}{2^q}\big)\xi_\ss(m)\frac{p_m(t)p_m(1)}{\|p_m\|_{w_d}^2}=\sul_{m=2^{q-1}+1}^{2^{q+1}}\zeta\big(\frac{m}{2^q}\big)\xi_\ss(m)\frac{P_m(t)P_m(1)}{\|P_m\|_{w_d}^2}.
    \end{equation}
    We aim to apply Theorem \ref{thm_petru_xu} to $\theta=\arccos t$ and $\phi=0$. To this end, we analyze the function
    \begin{equation*}
        a(u)=\zeta(u)\xi_\ss(2^qu).
    \end{equation*}
    In this case, for any $\beta\in\NN$, we can write 
    \begin{equation}
        \begin{split}
            \int_{\frac{1}{2}}^2|a^{(\beta)}(u)|du=&\int_{\frac{1}{2}}^2\Bigg|\sul_{\nu=0}^\beta\binom{\beta}{\nu}\zeta^{(\beta-\nu)}(u)\Big(\frac{d}{du}\Big)^\nu(\xi_\ss(2^qu))\Bigg|du\\
            \leq&2^\beta\int_{\frac{1}{2}}^2\max\limits_{0\leq\nu\leq\beta}\big|\zeta^{(\beta-\nu)}(u)2^{q\nu}\xi_\ss^{(\nu)}(2^qu)\big|du
            \lesssim\max\limits_{0\leq\nu\leq\beta}2^{q\nu}\int_{\frac{1}{2}}^2|\xi_\ss^{(\nu)}(2^qu)|du.
        \end{split}
    \end{equation}
    By Lemma \ref{cor_mu_esti}, $|\xi_\ss^{(\nu)}(2^qu)|\lesssim2^{-q(d+2k+1-2\ss)-\nu}$, and consequently
    \begin{equation}
        \|a^{(\beta)}\|_{\mathcal{L}^1}\lesssim2^{-q(d+2k+1-2\ss)},
    \end{equation}
    where the corresponding constant is independent of $q$.

    As a result, Theorem \ref{thm_petru_xu} (with $\theta=\arccos t$ and $\phi=0$) implies for any $\gamma\in\NN$,
    \begin{equation*}
        \begin{split}
            |\varphi_{\ss,q}(t)|\lesssim&\frac{2^{q\frac{d+1}{2}}\mal_{1\leq\beta\leq\gamma+2d-1}\|a^{(\beta)}\|_{\mathcal{L}^1}}{(\sqrt{1-t^2+2^{-2q}})^{\frac{d-1}{2}}(1+2^q\sqrt{1-t})^{\gamma}}\\
            \lesssim&\frac{2^{-q(\frac{d+1}{2}+2k+\gamma-2\ss)}}{(\sqrt{1+t}+2^{-q})^{\frac{d-1}{2}}(\sqrt{1-t}+2^{-q})^{\frac{d-1}{2}+\gamma}}.
        \end{split}
    \end{equation*}
    Same argument holds true to $\varphi_q$.
\end{proof}
Given this lemma, we can analyze the matrices.
\begin{lemma}\label{lem_QQs_norm}
    Let $\gamma>\frac{d+1}{2}$, then for all $q\in\NN$,
        \begin{equation}\label{eqn_QQs_upper}
            \begin{split}
                &\|Q(q)\|_2\lesssim\left\{\begin{array}{ll}
            \displaystyle2^{-q(2k+1)} & \qquad q\leq\log_2(\frac{1}{\underline{h}}) \\
            \displaystyle2^{-q(\frac{d+1}{2}+2k+\gamma)}\underline{h}^{-(\frac{d-1}{2}+\gamma)}, &\qquad q>\log_2(\frac{1}{\underline{h}})
        \end{array}\right.\\
        &\|Q_\ss(q)\|_2\lesssim\left\{\begin{array}{ll}
            \displaystyle2^{-q(2k-2s+1)} & \qquad q\leq\log_2(\frac{1}{\underline{h}}) \\
            \displaystyle2^{-q(\frac{d+1}{2}+2k-2s+\gamma)}\underline{h}^{-(\frac{d-1}{2}+\gamma)}, &\qquad q>\log_2(\frac{1}{\underline{h}})
        \end{array}\right.
            \end{split}
        \end{equation}
        Moreover, there exists $C>0$ such that for all $q\geq\log_2(\frac{C}{\underline{h}})$,
        \begin{equation}\label{eqn_Q_norm_largeq}
            \|Q(q)\|_2\simeq2^{-q(2k+1)},\quad\|Q_\ss(q)\|_2\simeq2^{-q(2k-2\ss+1)}.
        \end{equation}
\end{lemma}
\begin{proof}
For $q=0$, the estimation is obvious. We consider $q\geq1$. Following the idea in \cite{liu2025achieving}, we divide the set $\{\eta_i^*:~1\leq i\leq 2n\}$ in terms of the distance to $\eta_j^*$ and $-\eta_j^*$ as
\begin{equation*}
    \{\eta_i^*:~1\leq i\leq n\}=\bigcup\limits_{\nu=-1}^{\lceil\log_2\left(\frac{\pi}{2\underline{h}}\right)\rceil}\left(\mathcal{I}_{\nu,j,+}\cup\mathcal{I}_{\nu,j,-}\right),
\end{equation*}
where
$$\mathcal{I}_{-1,j,-}:=\left\{i:\rho(\eta_i^*,-\eta_j^*)<\underline{h}\right\},\quad\mathcal{I}_{-1,j,+}:=\left\{i:\rho(\eta_i^*,\eta_j^*)<\underline{h}\right\}$$
and for $\nu=0,1,\dots$,
$$\mathcal{I}_{\nu,j,+}:=\{i:2^\nu \underline{h}\leq\rho(\eta_i^*,\eta_j^*)<2^{\nu+1}\underline{h}\},\quad\mathcal{I}_{\nu,j,-}:=\{i:2^\nu \underline{h}\leq\rho(\eta_i^*,-\eta_j^*)<2^{\nu+1}\underline{h}\}.$$
By a measure argument, it is easy to verify
\begin{equation}\label{eqn_est_cardi_I}
    \begin{split}
        &\#\mathcal{I}_{-1,j,+}\lesssim1,\quad\#\mathcal{I}_{-1,j,-}\lesssim1\\
        &\#\mathcal{I}_{\nu,j,+}\simeq2^{\nu d},\quad\#\mathcal{I}_{\nu,j,-}\simeq2^{\nu d},\qquad\nu=0,\dots,\big\lceil\log_2\left(\frac{\pi}{2\underline{h}}\right)\big\rceil,
    \end{split}
\end{equation}
where the corresponding constants are only dependent of $d$. By noticing the formula
$$\sqrt{1-\eta_i^*\cdot \eta_j^*}=\sqrt{1-\cos(\rho(\eta_i^*,\eta_j^*))}=\sqrt{2}\sin\left(\frac{\rho(\eta_i^*,\eta_j^*)}{2}\right),\qquad \eta_i^*\cdot \eta_j^*\geq0,$$

we have
$$\sqrt{1-\eta_i^*\cdot \eta_j^*}\simeq\rho(\eta_i^*,\eta_j^*).$$
By Lemma \ref{lem_local_poly},
\begin{equation*}
    \left|R(q)_{i,j}\right|\lesssim\frac{2^{-q(\frac{d+1}{2}+2k+\gamma)}}{(\rho(\eta_i^*,-\eta_j^*)+2^{-q})^{\frac{d-1}{2}}(\rho(\eta_i^*,\eta_j^*)+2^{-q})^{\frac{d-1}{2}+\gamma}},
\end{equation*}
Since $\gamma>\frac{d+1}{2}$, we have
\begin{equation}\label{eqn_R_sum_ineqj}
    \begin{split}
        &\sul_{i\neq j} |R(q)_{i,j}| \lesssim \sul_{\nu=-1}^{\lceil\log_2\lrt{\frac{\pi}{2\underline{h}}}\rceil} \sul_{i\in\mathcal{I}_{\nu,j,+}\cup\mathcal{I}_{\nu,j,-}, i \neq j} \frac{2^{-q(\frac{d+1}{2}+2k+\gamma)}}{(\rho(\eta_i^*,-\eta_j^*)+2^{-q})^{\frac{d-1}{2}}(\rho(\eta_i^*,\eta_j^*)+2^{-q})^{\frac{d-1}{2}+\gamma}} \\
        \lesssim& \sul_{\nu=0}^{\lceil\log_2\left(\frac{\pi}{2\underline{h}}\right)\rceil} \bigg[\frac{(\#\mathcal{I}_{\nu,j,+}) 2^{-q(\frac{d+1}{2}+2k+\gamma)}}{(1+2^{-q})^{\frac{d-1}{2}}(2^\nu\underline{h}+2^{-q})^{\frac{d-1}{2}+\gamma}} + \frac{(\#\mathcal{I}_{\nu,j,-}) 2^{-q(\frac{d+1}{2}+2k+\gamma)}}{(2^\nu\underline{h}+2^{-q})^{\frac{d-1}{2}}(1+2^{-q})^{\frac{d-1}{2}+\gamma}}\bigg] \\
        &+ \frac{ (\#\mathcal{I}_{-1,j,+})2^{-q(\frac{d+1}{2}+2k+\gamma)}}{(1+2^{-q})^{\frac{d-1}{2}}(\underline{h}+2^{-q})^{\frac{d-1}{2}+\gamma}} + \frac{(\#\mathcal{I}_{-1,j,-}) 2^{-q(\frac{d+1}{2}+2k+\gamma)}}{(\underline{h}+2^{-q})^{\frac{d-1}{2}}(1+2^{-q})^{\frac{d-1}{2}+\gamma}} \\
        \lesssim& \sul_{\nu=0}^{\lceil\log_2\left(\frac{\pi}{2\underline{h}}\right)\rceil} \frac{2^{\nu d} 2^{-q(\frac{d+1}{2}+2k+\gamma)}}{(2^\nu\underline{h}+2^{-q})^{\frac{d-1}{2}+\gamma}} + \frac{2^{-q(\frac{d+1}{2}+2k+\gamma)}}{(\underline{h}+2^{-q})^{\frac{d-1}{2}+\gamma}}
         <\sul_{\nu=0}^{\lceil\log_2\left(\frac{\pi}{2\underline{h}}\right)\rceil}  \frac{2^{\nu d}2^{-q(\frac{d+1}{2}+2k+\gamma)}}{(2^\nu\underline{h})^{\frac{d-1}{2}+\gamma}} + \frac{2^{-q(\frac{d+1}{2}+2k+\gamma)}}{\underline{h}^{\frac{d-1}{2}+\gamma}}\\
        \lesssim&2^{-q(\frac{d+1}{2}+2k+\gamma)}\underline{h}^{-(\frac{d-1}{2}+\gamma)}.
    \end{split}
\end{equation}
and
\begin{equation*}
    |R(q)_{j,j}|\lesssim2^{-q(2k+1)}.
\end{equation*}
Thus
\begin{equation*}
    \sul_{i=1}^{2n} |R(q)_{i,j}| \lesssim\left\{\begin{array}{ll}
            \displaystyle2^{-q(2k+1)}, & \qquad q\leq\log_2(\frac{1}{\underline{h}}), \\
            \displaystyle2^{-q(\frac{d+1}{2}+2k+\gamma)}\underline{h}^{-(\frac{d-1}{2}+\gamma)}, &\qquad q>\log_2(\frac{1}{\underline{h}}).
        \end{array}\right.
\end{equation*}

As $R(q)$ is symmetric, this bound also holds to $\sul_{j=1}^{2n} |R(q)_{i,j}|$ for any $i$. Consequently, the norm of $R(q)$ is estimated as
\begin{equation*}
    \begin{split}
        \|R(q)\|_2 
        \leq &\sqrt{\|R(q)\|_1 \|R(q)\|_\infty}
        = \sqrt{ \max_{1\leq j \leq 2n} \sul_{i=1}^{2n} |R(q)_{i,j}| } 
           \sqrt{ \max_{1\leq i \leq 2n} \sul_{j=1}^{2n} |R(q)_{i,j}| } \\
        \lesssim&\left\{\begin{array}{ll}
            \displaystyle2^{-q(2k+1)}, & \qquad q\leq\log_2(\frac{1}{\underline{h}}), \\
            \displaystyle2^{-q(\frac{d+1}{2}+2k+\gamma)}\underline{h}^{-(\frac{d-1}{2}+\gamma)}, &\qquad q>\log_2(\frac{1}{\underline{h}}).
        \end{array}\right.
    \end{split}
\end{equation*}
Similarly,
\begin{equation*}
    \|R_s(q)\|_2\lesssim\left\{\begin{array}{ll}
            \displaystyle2^{-q(2k-2s+1)}, & \qquad q\leq\log_2(\frac{1}{\underline{h}}), \\
            \displaystyle2^{-q(\frac{d+1}{2}+2k-2s+\gamma)}\underline{h}^{-(\frac{d-1}{2}+\gamma)}, &\qquad q>\log_2(\frac{1}{\underline{h}}).
        \end{array}\right.
\end{equation*}
Together with \eqref{eqn_rel_R_Q} and \eqref{eqn_rel_Rs_Qs} proves \eqref{eqn_QQs_upper}.

For the lower bound, we invoke the diagonal dominance property of $R(q)$ and $R_s(q)$. Clearly, all the diagonal values of $R(q)$ equals to
\begin{equation*}
    R(q)_{i,i}=\varphi_q(\eta_i^*\cdot\eta_i^*)=\varphi_q(1),\qquad i=1,\dots,2n.
\end{equation*}
To estimate this value, we recall $\zeta(t)>c_1>0$ for $t\in[3/5,5/3]$ and the value of the $p_m(1)$ (see, e.g., \cite[(3.7)]{liu2025achieving})
\begin{equation}
    p_m(1)=N(m)\simeq m^{d-1},\qquad m=1,2,\dots.
\end{equation}
Thus
\begin{equation*}
    \begin{split}
        R(q)_{i,i}=&\varphi_{q}(1)= \sul_{m=2^{q-1}+1}^{2^{q+1}}\zeta\big(\frac{m}{2^q}\big)\xi(m)p_m(1)\geq\sul_{\frac{3}{5}2^q\leq m\leq\frac{5}{3}2^q}c_1\xi(m)p_m(1)\\
        \gtrsim&\sul_{\frac{3}{5}2^q\leq m\leq\frac{5}{3}2^q}2^{-q(d+2k+1)}2^{q(d-1)}
        \gtrsim2^{-q(2k+1)}.
    \end{split}
\end{equation*}
Together with \eqref{eqn_R_sum_ineqj}, there exists a constant $C$, such that whenever $2^q\geq C\underline{h}^{-1}$,
\begin{equation*}
    R(q)_{i,i}\geq2\sul_{i\neq j} |R(q)_{i,j}|.
\end{equation*}
As a result,
\begin{equation}\label{eqn:dig_dom_matrix}
    \begin{split}
        R(q) - \left( \frac{1}{2} R(q)_{i,i} \right) I_{2n\times2n} &\succeq 0, \\
        \left( \frac{3}{2} R(q)_{i,i} \right) I_{2n\times2n} - R(q) &\succeq 0.
    \end{split}
\end{equation}
Therefore, for any $a\in\RR^n$,
\begin{equation*}
    \begin{split}
        a^\top Q(q)a=&(a^\top,a^\top)R(q)\binom{a}{a}\geq\frac{1}{2} R(q)_{i,i}(a^\top,a^\top)I_{2n\times2n}\binom{a}{a}\gtrsim2^{-q(k+1)}\|a\|_2^2,\\
        a^\top Q(q)a=&(a^\top,a^\top)R(q)\binom{a}{a}\leq\frac{3}{2} R(q)_{i,i}(a^\top,a^\top)I_{2n\times2n}\binom{a}{a}\lesssim2^{-q(k+1)}\|a\|_2^2.
    \end{split}
\end{equation*}
Same argument can be applied to $Q_\ss(q)$ This proves \eqref{eqn_Q_norm_largeq}.
\end{proof}

\subsection{Estimating the condition number by needlet matrices}

%To finish the proof, 
\begin{proof}[Proof of Theorem \ref{thm:eigenvalue}]

The upper bound of $\lambda_{\mathrm{max}}(M)$ follows directly from the matrix norm inequality:
\begin{equation}\label{eqn:l_max_upper}
    \begin{split}
        \lambda_{\mathrm{max}}(M) &= \|M\|_2 
        \leq \sqrt{\|M\|_1 \|M\|_\infty}
        = \sqrt{ \max_{1\leq j \leq n} \sul_{i=1}^n M_{i,j} } 
          \cdot \sqrt{ \max_{1\leq i \leq n} \sul_{j=1}^n M_{i,j} } \\
        &\leq \sqrt{ \sul_{i=1}^n 1 } \cdot \sqrt{ \sul_{i=1}^n 1 } = n \lesssim \underline{h}^{-d}.
    \end{split}
\end{equation}

% The same argument applies to the stiffness matrix $K_s$, giving
% \begin{equation}
%     \lambda_{\mathrm{max}}(K_s) \lesssim \underline{h}^{-d}.
% \end{equation}

To establish the lower bound for $\lambda_{\mathrm{max}}$, we construct a specific test vector $a\in\RR^n$ and apply the Rayleigh quotient. Define
\[
a = \left( Y_{k+1,1}(\theta^*_1)\tau_1,\dots,Y_{k+1,1}(\theta^*_n)\tau_n \right)^\top,
\]
where $\{\tau_j\}_{j=1}^n$ are the quadrature weights provided by Lemma~\ref{lem:quadrature}. Then, using the orthogonality of spherical harmonics and the expansion of $M$, we obtain

\begin{equation*}
    \begin{split}
        a^\top M a 
        &= \sul_{m=0}^\infty \widehat{\sigma_k}(m)^2 a^\top P(m) a 
        \geq \widehat{\sigma_k}(k+1)^2 a^\top P(k+1) a \\
        &= \widehat{\sigma_k}(k+1)^2 \sul_{\ell=1}^{N(k+1)} a^\top 
        \left( Y_{k+1,\ell}(\theta_i^*) Y_{k+1,\ell}(\theta_j^*) \right)_{i,j=1}^n a \\
        &= \widehat{\sigma_k}(k+1)^2 \sul_{\ell=1}^{N(k+1)} 
        \left( \sul_{j=1}^n Y_{k+1,1}(\theta_j^*) Y_{k+1,\ell}(\theta_j^*) \tau_j \right)^2 \\
        &= \widehat{\sigma_k}(k+1)^2 \sul_{\ell=1}^{N(k+1)} 
        \left( \fint_{\SS^d} Y_{k+1,1}(\theta) Y_{k+1,\ell}(\theta)\, d\theta \right)^2 \\
        &= \widehat{\sigma_k}(k+1)^2.
    \end{split}
\end{equation*}

Meanwhile, using the bound $\tau_j \lesssim h^d$, we estimate the norm of $a$ as
\begin{equation*}
    \|a\|_2^2 = \sul_{j=1}^n Y_{k+1,1}(\theta_j^*)^2 \tau_j^2 
    \lesssim h^d \sul_{j=1}^n Y_{k+1,1}(\theta_j^*)^2 \tau_j = h^d.
\end{equation*}

Therefore, we conclude that
\begin{equation}
    \lambda_{\mathrm{max}}(M) \geq \frac{a^\top M a}{\|a\|_2^2} \gtrsim h^{-d}.
\end{equation}

% A similar argument applies to the stiffness matrix. Using the same test vector $a$ and the spectral expansion of $K_s$ in \eqref{eqn:K_decomp}, we have

% \begin{equation*}
%     \begin{split}
%         a^\top K_s a 
%         &= \sul_{m\in E_s}^\infty 
%         \xi_s(m)
%         a^\top P(m) a \\
%         &\geq \xi_s(k+1) a^\top P(k+1) a= \xi_s(k+1).
%     \end{split}
% \end{equation*}

% Hence,
% \begin{equation}
%     \lambda_{\mathrm{max}}(K) \geq \frac{a^\top K_s a}{\|a\|_2^2} \gtrsim h^{-d}.
% \end{equation}

Consider $\lambda_{\min}$. By Lemma \ref{lem_QQs_norm}, for any $a\in\RR^n$
    \begin{equation*}
        \begin{split}
            a^\top Ma\geq\sul_{q=\lceil\log_2(\frac{C}{\underline{h}})\rceil}^\infty a^\top Q(q)a\simeq\sul_{q=\lceil\log_2(\frac{C}{\underline{h}})\rceil}^\infty2^{-q(2k+1)}\|a\|_2^2\simeq\underline{h}^{2k+1}\|a\|_2^2,
        \end{split}
    \end{equation*}
    which implies
    \begin{equation}
        \lambda_{\mathrm{min}}(M) \gtrsim \underline{h}^{2k+1}.
    \end{equation}
    For the upper bound of $\lambda_{\mathrm{min}}(M)$, we suppress the low-degree contributions and construct a specific test vector supported in a high-frequency space. Let
\[
a = \big( Y_{J,1}(\theta^*_1)\tau_1, \dots, Y_{J,1}(\theta^*_n)\tau_n \big)^\top,
\]
where $J = \lfloor C_1 h^{-1} \rfloor$ and the weights $\{\tau_j\}$ are those from Lemma~\ref{lem:quadrature}. Then for $m \leq J$, by the orthogonality condition \eqref{eqn:quadrature_formula}, we have
\begin{equation}\label{eqn:orthogonal_m<k+1}
    \begin{split}
        a^\top P(m) a 
        &= \sul_{i,j=1}^n a_i a_j Y_{m,\ell}(\theta_i^*) Y_{m,\ell}(\theta_j^*) 
        = \sul_{\ell=1}^{N(m)} \left( \sul_{j=1}^n Y_{m,\ell}(\theta_j^*) a_j \right)^2 \\
        &= \sul_{\ell=1}^{N(m)} \left( \sul_{j=1}^n Y_{m,\ell}(\theta_j^*) Y_{J,1}(\theta_j^*) \tau_j \right)^2 
        = \sul_{\ell=1}^{N(m)} \left( \int_{\SS^d} Y_{m,\ell}(\theta) Y_{J,1}(\theta)\, d\theta \right)^2 = 0.
    \end{split}
\end{equation}

Hence
\begin{equation*}
    a^\top Q(q)a=0,\qquad q=0,1,\dots,\lfloor\log_2J\rfloor-1.
\end{equation*}
Consequently,
    \begin{equation*}
        \begin{split}
            a^\top Ma\leq&\sul_{q=\lfloor\log_2J\rfloor}^\infty a^\top Q(q)a\\
            \lesssim&\sul_{q=\lfloor\log_2J\rfloor}^{\lfloor\log_2(\frac{1}{\underline{h}})\rfloor}2^{-q(2k+1)}\|a\|_2^2+\sul_{q=\lfloor\log_2(\frac{1}{\underline{h}})\rfloor+1}^\infty2^{-q(\frac{d+1}{2}+2k+\gamma)}\underline{h}^{-(\frac{d-1}{2}+\gamma)}\|a\|_2^2
            \lesssim h^{2k+1}\|a\|_2^2.
        \end{split}
    \end{equation*}
    This implies
    \begin{equation}
        \underline{h}^{2k+1}\lesssim\lambda_{\min}\lesssim h^{2k+1}.
    \end{equation}

A similar analysis holds for the stiffness matrix $K_s$. The same proof technique applies, with the matrices $Q(q)$ replaced by $Q_s(q)$ from Lemma~\ref{lem_QQs_norm}.

\end{proof}

\subsection{Spectral analysis by spectral of weighted matrices}

In this section, we normalize $M$ and $K_s$ with the weights in Lemma \ref{lem:quadrature} and provide a more details spectral analysis to the achieved new matrices:
\begin{equation}
    M_\tau:=W^{1/2}MW^{1/2},\quad K_{\ss,\tau}:=W^{1/2}MW^{1/2},
\end{equation}
where $W=\mathrm{diag}(\tau_1,\dots,\tau_n)$.

Our theorem in this section provides perfect characterization of the spectral of $M_\tau$ and $K_{\ss,\tau}$. However, it does not necessarily imply the spectral properties of $M$ and $K_\ss$ since $\tau_j$ does not have a lower bound in general. It is known that, in some special case (e.g. \cite{bondarenko2013optimal}), the property of $\tau_j$'s are good and then we can achieve the spectral properties of the $M$ and $K_\ss$.

% \begin{theorem}\label{thm_spectr}
%     Let
%     $$d_m:=\sul_{i\in E,~i\in m}N(i),$$
%     then there exists some $J_0$ and $C_2$ independent of $n$ such that for $m=J_0,\dots,J$, and $m\in E$,
%     \begin{equation}
%         \lambda_{d_m}(M_\tau)\geq\xi(m),\qquad \lambda_{d_m+1}\leq\max\{\xi(m+2),C_2h^{2k+1}\}.
%     \end{equation}
%     In this case, if $m$ satisfies $\xi(m+2)\geq C_2h^{2k+1}$, we have
%     \begin{equation}
%         \xi(m)\geq\lambda_{d_{m-1}+1}\geq\dots\geq\lambda_{d_{m}}\geq\xi(m+2)
%     \end{equation}
% \end{theorem}

\begin{theorem}\label{thm_spectr}
Let $d,k,n,J$, $\{\theta_j^*\}_{j=1}^n$, and $\{d_m\}_{m=0}^\infty$ be as in Theorem \ref{thm_spectr_main}. There exists some $J_0$ and $C_2$ independent of $n$ such that for $m$ satisfying:
    \begin{eqnarray*}
        &&m+2\in E\cap\{J_0,J_0+1,\dots,J\},\\
        &&\xi(m+2)\geq C_2h^{2k+1},\\
        &&\xi_s(m+2)\geq C_2h^{2(k-s)+1},
    \end{eqnarray*}
    we have
    \begin{equation}\label{eqn_spectr_main}
        \begin{split}
            &\xi(m+2)+C_2h^{d+2k+1}\geq\lambda_{d_{m}+1}(M_\tau)\geq\dots\geq\lambda_{d_{m+2}}(M_\tau)\geq\xi(m+2),\\
            &\xi_\ss(m+2)+C_2h^{d+2(k-s)+1}\geq\lambda_{d_{m}+1}(K_{\ss,\tau})\geq\dots\geq\lambda_{d_{m+2}}(K_{\ss,\tau})\geq\xi_\ss(m+2).
        \end{split}
    \end{equation}
\end{theorem}

\begin{proof}%[Proof of Theorem \ref{thm_spectr}]
    Let $E(m+2):=E\cap\{0,\dots,m+2\}$, we define vectors
    \begin{equation}\label{eqn_def_Vml}
        V_{i,\ell}=(\sqrt{\tau_1}Y_{i,\ell}(\theta_1^*),\dots,\sqrt{\tau_n}Y_{i,\ell}(\theta_n^*))^\top,\qquad 1\leq\ell\leq N(i),~i\in E(J)
    \end{equation}
    and linear spaces
    \begin{equation}
        \begin{split}
            \mathcal{U}(i):=\mathrm{span}\big\{V_{i,\ell}:~1\leq\ell\leq N(i)\big\},\qquad i\in E(J).
        \end{split}
    \end{equation}
    By Lemma \ref{lem:quadrature}, the vectors $V_{m,\ell}$ in \eqref{eqn_def_Vml} are orthonormal:
    \begin{equation*}
        V_{i,\ell}^\top V_{i',\ell'}=\sul_{j=1}^n\tau_jY_{i,\ell}(\theta_j^*)Y_{i',\ell'}(\theta_j^*)=\int_{\SS^d}Y_{i,\ell}(\eta)Y_{i',\ell'}(\eta)d\eta=\delta_{(i,\ell),(i',\ell')}.
    \end{equation*}
    With $m+2\in E(J)$, for any $v\in\bigoplus\limits_{i\in E(m+2)}\mathcal{U}(i)$, we can write
    \begin{equation}
        v=\sul_{{i\in E(m+2)}}\sul_{\ell=1}^{N(i)}v_{i,\ell}V_{i,\ell},
    \end{equation}
    where
    \begin{equation*}
        v_{i,\ell}=v_{i,\ell}^\top V_{i,\ell},\qquad1\leq\ell\leq N(i),~i\in E(J).
    \end{equation*}
    Now
    \begin{equation}
        \begin{split}
            M_\tau=&\sul_{i\in E}\xi(i)W^{1/2}P(i)W^{1/2}=\sul_{i\in E}\xi(i)\sul_{\ell=1}^{N(i)}W^{1/2}\begin{pmatrix}
            Y_{i,\ell}(\theta_1^*)\\
            \vdots\\
            Y_{i,\ell}(\theta_1^*)
        \end{pmatrix}\big(Y_{i,\ell}(\theta_1^*),\dots,Y_{i,\ell}(\theta_n^*)\big)W^{1/2}\\
        =&\sul_{i\in E}\sul_{\ell=1}^{N(i)}V_{i,\ell}V_{i,\ell}^\top.
        \end{split}
    \end{equation}

    Let $J_0\geq k+1$ be large enough that
    $$\xi(J_0)\leq\min\Big\{\widehat{\sigma_k}(0)^2,\dots,\widehat{\sigma_k}(k)^2\Big\},$$
    then
    \begin{equation*}
        \begin{split}
            v^\top M_\tau v\geq&\sul_{i\in E(m+2)}\xi(i)\sul_{\ell=1}^{N(i)}v^\top V_{i,\ell}V_{i,\ell}^\top v\\
            =&\sul_{i\in E(m+2)}\xi(i)\sul_{\ell=1}^{N(i)}v_{i,\ell}^2\geq\xi(m+2)\|v\|_2^2,
        \end{split}
    \end{equation*}
    With \eqref{eqn_def_dm},
    $$d_{m+2}=\sul_{i\in E(m+2)}N(i)=\mathrm{dim}\Big(\bigoplus\limits_{i\in E(m+2)}\mathcal{U}(i)\Big),$$
    we can apply Courant–Fischer Theorem and conclude
    \begin{equation}
        \lambda_{d_{m+2}}=\mal_{\substack{\mathcal{U}\subset\RR^n\\ \mathrm{dim}(\mathcal{U})=d_{m+2}}}\mil_{v\in\mathcal{U},\|v\|_2=1}v^\top M_\tau v\geq\mil_{\substack{v\in\bigoplus\limits_{i\in E(m+2)}\mathcal{U}(i)\\
        \|v\|_2=1}}v^\top M_\tau v=\xi(m+2).
    \end{equation}

    On the other hand, we consider $v\in\mathcal{V}(m):=\Big(\bigoplus\limits_{i\in E(m)}\mathcal{U}(i)\Big)^\perp$. Then we have the decomposition
    \begin{equation*}
        v=\sul_{i=m+2}^J\sul_{\ell=1}^{N(i)}v_{i,\ell}V_{i,\ell}+u,
    \end{equation*}
    where $u\in\Big(\bigoplus\limits_{i\in E(J)}\mathcal{U}(i)\Big)^\perp$.
    Then
    \begin{equation*}
        \begin{split}
            \mal_{\substack{v\in\mathcal{V}(m)\\ \|v\|_2=1}}v^\top M_\tau v\leq&\sul_{i\in E(J)}\xi(i)\sul_{\ell=1}^{N(i)}v^\top V_{i,\ell}V_{i,\ell}^\top v+v^\top W^{1/2}\Big(\sul_{i\in E\setminus E(J)}\xi(i)P(i)\Big)W^{1/2}v\\
            \leq&\sul_{i\in E(J)\setminus E(m)}\xi(i)\sul_{\ell=1}^{N(i)}v_{i,\ell}^2+\sul_{q=\lfloor\log_2J\rfloor-1}^\infty (v^\top W^{1/2})Q(q)(W^{1/2}v).
        \end{split}
    \end{equation*}

    By Lemma \ref{lem_QQs_norm},
    \begin{equation*}
        \begin{split}
            &\sul_{q=\lfloor\log_2J\rfloor-1}^{\infty}(v^\top W^{1/2})Q(q)(W^{1/2}v)\\
            \lesssim&\sul_{q=\lfloor\log_2J\rfloor-1}^{\lfloor\log_2(\frac{1}{\underline{h}})\rfloor}2^{-q(2k+1)}\|vW^{1/2}\|_2^2+\sul_{q=\lfloor\log_2(\frac{1}{\underline{h}})\rfloor+1}^\infty2^{-q(\frac{d+1}{2}+2k+\gamma)}\underline{h}^{-(\frac{d-1}{2}+\gamma)}\|vW^{1/2}\|_2^2\\
            \lesssim&h^{2k+1}\|vW^{1/2}\|_2^2=h^{2k+1}\sul_{j=1}^n\tau_jv_j^2\lesssim h^{d+2k+1}\sul_{j=1}^nv_j^2.
        \end{split}
    \end{equation*}
    Let $C_2$ be the corresponding constant, i.e.,
    \begin{equation}\label{eqn_sum_geqJ}
        \sul_{q=\lfloor\log_2J\rfloor-1}^\infty(v^\top W^{1/2})Q(q)(W^{1/2}v)\leq C_2h^{d+2k+1}\|v\|_2^2,
    \end{equation}
    then
    \begin{equation*}
        v^\top Mv\leq\big(\xi(m+2)+C_2h^{d+2k+1}\big)\|v\|_2^2
    \end{equation*}

    Again, Courant–Fischer Theorem gives
    \begin{equation}
        \lambda_{d_m+1}=\mil_{\substack{\mathcal{U}\subset\RR^n\\ \mathrm{dim}(\mathcal{U})=n-d_m}}\mal_{v\in\mathcal{U},\|v\|_2=1}v^\top M_\tau v\leq\mal_{v\in\mathcal{V}(m),\|v\|_2=1}v^\top M_\tau v\leq\xi(m+2)+C_2h^{d+2k+1}.
    \end{equation}

\end{proof}

\begin{proof}[Proof of Theorem \ref{thm_spectr_main}]
When $\tau_1=\dots=\tau_n=n^{-1}$, \eqref{eqn_spectr_main2} follows immediately by Theorem \ref{thm_spectr}. It suffices to prove (1) of Theorem \ref{thm_spectr_main}. In this case, we have
\begin{equation*}
    h\simeq\underline{h}\simeq J^{-1}\simeq n^{-\frac{1}{d}}.
\end{equation*}
Without loss of generality, assume $J_0\in E$. Then for $j\leq d_{J_0}$, we have
\begin{equation*}
    \lambda_j(M_\tau)\geq\lambda_{d_{J_0}}(M_\tau)\geq\xi(J_0).
\end{equation*}

Next, for $d_m+1\leq j\leq d_{m+2}$ with $m+2\in E\cap\{J_0,\dots,J\}$, we can write \eqref{eqn_spectr_main} as
\begin{equation*}
    \xi(m+2)+C_2h^{d+2k+1}\geq\lambda_{d_{m}+1}\geq\lambda_j(M_\tau)\geq\lambda_{d_{m+2}}(M_\tau)\geq\xi(m+2)
\end{equation*}
Since $J\simeq h^{-1}$, by Lemma \ref{lem:Bach}, we have
\begin{equation*}
    \xi(m+2)+C_2h^{d+2k+1}\lesssim m^{-(d+2k+1)}+J^{-(d+2k+1)}\simeq m^{-(d+2k+1)}\simeq\xi(m+2).
\end{equation*}
    By \eqref{eqn_def_Nm} and \eqref{eqn_def_dm}, we have
    \begin{equation*}
        m^d\simeq d_m+1\leq j\leq d_{m+2}\simeq m^d.
    \end{equation*}
    Thus
    \begin{equation*}
        \lambda_j(M_\tau)\simeq m^{-(d+2k+1)}\simeq j^{-\frac{d+2k+1}{d}}.
    \end{equation*}
    For $j>d_{J-2}$, we have
    \begin{equation*}
        \lambda_{d_j}\leq\lambda{d_{J-2}+1}\leq\xi(J)+C_2h^{d+2k+1}\lesssim h^{d+2k+1}.
    \end{equation*}
    Since $h\simeq\underline{h}\leq\mil_{i\neq j}\rho(\theta_i^*,\theta_j^*)$, by Lemma \ref{lem:quadrature},
    \begin{equation*}
        \tau_j\simeq n^{-1},\qquad j=1,\dots,n.
    \end{equation*}
    With $d_{J-2}\simeq J^d\simeq n$, we can combine the above estimations and conclude
    \begin{equation}
        \lambda_j(M)\simeq n\lambda_j(M_\tau)\simeq nj^{-\frac{d+2k+1}{d}},\qquad j=1,\dots,n.
    \end{equation}
    where the upper bound of $\lambda_j(M)$ for $j\leq d_{J-1}$ and the lower bound for $j>d_{J-1}$ is given in Theorem \ref{thm:eigenvalue}. This proves the \eqref{eqn_spectr_main1}. Similar arguments can be applied to the stiffness matrix $K_\ss$ and gives \eqref{eqn_spectr_main1_stiff}.

\end{proof}

% \begin{remark}
% A similar analysis holds for the stiffness matrix $K_s$. The same proof technique applies, with the coefficient sequence $\widehat{\sigma_k}(m)^2$ replaced by $\xi_s(m)$ from Lemma~\ref{cor_mu_esti}. This changes the eigenvalue decay rate to $m^{-(d+2(k-s)+1)}$, consistent with the condition number scaling in Theorem~\ref{thm:eigenvalue_stiff}.
% \end{remark}

%\input{least_square_random}
%\input{random-fearue-cond}

\section{Quasi-uniform Designs Ensure Stability Beyond\\ Random Sampling}\label{sec:random}

Several well–established randomized constructions fix the hidden-layer parameters independently from a prescribed distribution, among them stochastic basis selection~\cite{Igelnik1995}, Extreme Learning Machines (ELM)~\cite{Huang2006,huang2006universal,Liu2014}, and the random feature framework~\cite{rahimi2007random,rahimi2008weighted,bach2017equivalence,li2019towards,nelsen2021random,mei2022generalization}. Typical choices include uniform sampling on the sphere or Gaussian initialization~\cite{he2015delving}. These approaches have been investigated in a broad range of settings, from regression and classification to the numerical solution of PDEs~\cite{lopez2014randomized,belkin2020two,gerace2020generalisation,hu2022universality,dwivedi2020physics,dong2021local,chen2022bridging,zhang2024transferable,li2025local,chi2024random,sun2024domain,lu2025multiple}. On the approximation side, it is known that such randomized feature maps retain nearly optimal rates. For instance, one can show~\cite{liu2025achieving} that, with probability at least $1-\delta$,
\begin{equation}\label{eqn:rate_random}
    \inf_{f_n\in L_n^k}\|f-f_n\|_{\mathcal{L}^2(\Omega)}\lesssim\left(\frac{n}{\log(n/\delta)}\right)^{-\frac{r}{d}}\|f\|_{\mathcal{H}^r(\Omega)},
\end{equation}
which essentially matches the best achievable order. Thus, in terms of approximation power, random feature sampling competes with quasi-uniform constructions, and the key distinction lies in the stability properties of the resulting Gram matrices. We next quantify this difference through estimates of the fill distance $h$ and separation distance $\underline{h}$ for random samples on $\SS^d$.

\begin{lemma}
Let $\theta_1,\dots,\theta_n$ be i.i.d.\ samples from the uniform (surface-area) distribution on $\SS^d$, then
\begin{enumerate}
    \item [(1)] With probability at least $1-\delta$,
    \begin{equation}\label{eqn:rand_samp_quasi}
        h\lesssim\lrt{\frac{\log\frac{n}{\delta}}{n}}^{\frac{1}{d}}.
    \end{equation}
    \item [(2)] With probability at least $1-\delta$,
    \begin{equation}
        \underline{h}\lesssim\left(\frac{\log\frac{1}{\delta}}{n^2}\right)^{\frac{1}{d}}.
    \end{equation}
\end{enumerate}
\end{lemma}

\begin{proof}
let $N\in\NN$ be an integer to be determined later and divide $\SS^d$ into disjoint subsets
    \begin{equation*}
        \SS^d=\bigcup_{u=1}^NA_u
    \end{equation*}
    with
    \begin{equation}
        \fint_{A_u}1d\eta\simeq N^{-1},\quad\mathrm{diam}(A_u)\simeq N^{-\frac{1}{d}},\qquad u=1,\dots,N,
    \end{equation}
    where the corresponding constants are only dependent of $d$.
    
    \textbf{Proof of (1).} Now denote the event
    \begin{equation*}
        \mathcal{A}:=\left\{\hbox{There exists some $i$ such that }A_u\cap\{\theta_j\}_{j=1}^n=\varnothing\right\}.
    \end{equation*}
    Then there exists some $c=c(d)$ such that
    \begin{equation}\label{eqn:prob_A1}
        \begin{split}
            \mathrm{Prob}(\mathcal{A})\leq&\sul_{u=1}^N\mathrm{Prob}\lrt{A_u\cap\{\theta_j\}_{j=1}^n=\varnothing}=\sul_{u=1}^N\prod\limits_{j=1}^n\mathrm{Prob}\lrt{\theta_j\in\SS^d\setminus A_u}\\
            \leq&\sul_{u=1}^N(1-cN^{-1})^n\leq N\lrt{\frac{1}{e}}^{\frac{cn}{N}}.
        \end{split}
    \end{equation}
    Taking $\displaystyle N=\left\lfloor\frac{cn}{\log\frac{cn}{\delta}}\right\rfloor$, then \eqref{eqn:prob_A1} yields
    \begin{equation}
        \mathrm{Prob}(\mathcal{A})\leq\delta.
    \end{equation}
    That is, with probability at least $1-\delta$,
    \begin{equation*}
        A_u\cap\{\theta_j\}_{j=1}^n\neq\varnothing,\qquad u=1,\dots,N.
    \end{equation*}
    For any $\eta\in\SS^d$, there exists some $u$ such that $\theta\in A_u$. Let $\theta^*$ be a point in $A_u\cap\{\theta_j\}_{j=1}^n$, then
    \begin{equation*}
        \rho\lrt{\eta,\theta^*}\leq\mathrm{diam}(A_u)\lesssim N^{-\frac{1}{d}}\lesssim\lrt{\frac{\log\frac{n}{\delta}}{n}}^{\frac{1}{d}}.
    \end{equation*}
    Therefore, 
    \begin{equation}
        h\lesssim\lrt{\frac{\log\frac{n}{\delta}}{n}}^{\frac{1}{d}}.
    \end{equation}

    \textbf{Proof of (2).}
    Map each point $\theta_j$ to the index $u$ it falls into:
    $$\theta_j\mapsto u_j$$
    with $\theta_j\in A_{u_j}$. Now denote the event
    \begin{equation*}
        \mathcal B:=\left\{\hbox{the labels $u_1,\dots,u_n$ are distinct}\right\}.
    \end{equation*}
    If $\mathcal B$ fails, then some $i\neq j$ satisfy $u_i=u_j$, hence
    \[
      \rho(\theta_i,\theta_j)\le \mathrm{diam}(A_{u_i})\lesssim N^{-1/d}.
    \]

    Next we bound $\PP(\mathcal B)$. Let $p_u:=\PP(u_j=u)$ be the (normalized) surface measure of $A_u$; by construction $p_u\simeq N^{-1}$ and $\sum_{u=1}^N p_u=1$. Since the samples are i.i.d.,
    \[
      \PP(\mathcal B)=n!\!\!\sum_{1\le k_1<\cdots<k_n\le N}\! p_{k_1}\cdots p_{k_n}
      \;=\; n!\,e_n(p_1,\dots,p_N),
    \]
    where $e_n$ is the $n$th elementary symmetric polynomial. Maclaurin’s inequality (or the convexity/majorization argument that maximizes $e_n$ under $\sum p_u=1$ at $p_u\equiv 1/N$) yields
    \[
      e_n(p_1,\dots,p_N)\le \binom{N}{n}\Big(\frac{1}{N}\Big)^n,
    \]
    hence
    \[
      \PP(\mathcal B)\le \frac{N(N-1)\cdots(N-n+1)}{N^n}=\prod_{k=0}^{n-1}\Big(1-\frac{k}{N}\Big)
      \;\le\; \exp\!\Big(-\frac{1}{N}\sum_{k=0}^{n-1}k\Big)
      \;=\; \exp\!\Big(-\frac{n(n-1)}{2N}\Big).
    \]

    Now choose
    $$N=\left\lfloor\frac{n(n-1)}{2\log\frac{1}{\delta}}\right\rfloor,$$
    then we have $\PP(\mathcal B)\leq\delta$. That is, with probability at least $1-\delta$, there exists some $i\neq j$ such that $u_i=u_j$, which implies
    \begin{equation}
        \rho(\theta_i,\theta_j)\leq2\mathrm{diam}(A_{u_i})\lesssim N^{-\frac{1}{d}}\simeq\left(\frac{\log\frac{1}{\delta}}{n^2}\right)^{\frac{1}{d}}.
    \end{equation}
    This proves $\underline{h}\leq\mil_{i\neq j}\rho(\theta_i,\theta_j)\lesssim\left(\frac{\log\frac{1}{\delta}}{n^2}\right)^{\frac{1}{d}}.$

\end{proof}

This lemma shows that, with high probability, random feature sampling yields a fill distance $h$ of essentially the same order as in the quasi-uniform case, up to logarithmic factors. However, the separation distance $\underline{h}$ can be much smaller, of order $n^{-2/d}$ rather than $(n/\log n)^{-1/d}$. According to our main theorems, both the upper bounds for $\lambda_{\max}$ and the lower bounds for $\lambda_{\min}$ depend on $\underline{h}$. As a consequence, the condition numbers of $M$ and $K_s$ can deteriorate significantly in the random setting.

To illustrate the role of $\underline{h}$, consider the extreme case (an event of probability zero) where two parameters coincide, i.e.\ $\theta_i=\theta_j$ for some $i\neq j$. In this situation, the corresponding mass and stiffness matrices become singular and the condition number is undefined. Thus, although random feature methods preserve near-optimal approximation rates, they lack uniform guarantees on numerical stability: the mass and stiffness matrices $M$ and $K_s$ may in principle become severely ill-conditioned. This stands in sharp contrast to the quasi-uniform design, which ensures polynomially controlled condition numbers, and highlights the trade-off between the simplicity of random sampling and the guaranteed stability afforded by structured constructions.

\section{Numerical Experiments}\label{sec:numerics}

To empirically validate our theoretical findings, we conduct a numerical experiment to study the condition number growth of the mass matrix $M$. The basis functions are $\sigma_k(\theta \cdot x)$, with directions $\{\theta_j^*\}_{j=1}^n$ chosen to be quasi-uniform points on the sphere $\SS^d$.   The sets of quasi-uniform points $\{\theta_j^*\}_{j=1}^n \subset \SS^d$ used in our experiments are generated by numerically minimizing the Riesz $s$-energy functional. For a set of $n$ points $X_n = \{\theta_1, \dots, \theta_n\}$ on the sphere, the Riesz energy is defined as
\[
E_s(X_n) = \sum_{i \neq j} \|\theta_i - \theta_j\|^{-s}.
\]
Minimizing this energy is a well-known method for producing point sets that are asymptotically uniformly distributed and satisfy the quasi-uniformity condition $h \simeq \underline{h}$ \cite{saff1997distributing}. For our experiments, we choose $s=d$ and use a gradient-based optimization method to find approximate minimizers, which serve as our quasi-uniform parameter sets.

For various combinations of dimension $d$, activation order $k$, and number of neurons $n$, we assemble the mass matrix $M$ using quasi-Monte Carlo integration. We then compute its condition number $\kappa(M) = \lambda_{\max}(M) / \lambda_{\min}(M)$.

The results are presented in Figure~\ref{fig:condition_number_growth}, which shows the computed condition numbers $\kappa(M)$ plotted against the number of neurons $n$ on a log-log scale. The plots also include a reference line with the slope predicted by our main result, Theorem~\ref{thm:eigenvalue}. Specifically, the theoretical growth rate is $\mathcal{O}(n^{\alpha})$, where the exponent $\alpha$ is given by
\[
    \alpha = \frac{d + 2k + 1}{d}.
\]
As shown in the figure, the empirical data closely follows the theoretical slope across all tested configurations. This alignment provides strong numerical evidence for the sharpness of our condition number estimates for ReLU$^k$ bases on the sphere.

\begin{figure}[H]
        \centering
        \includegraphics[width=0.6\linewidth]{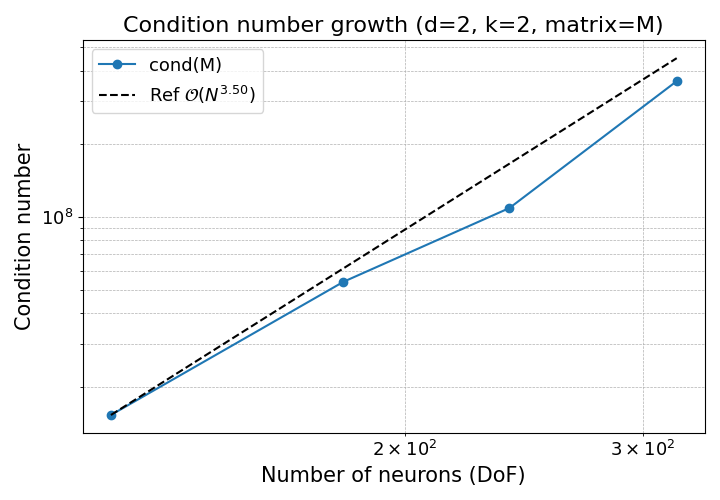}
        \caption{Condition number $\kappa(M)$ of the mass matrix versus the number of neurons $n$ for different dimensions $d$ and activation orders $k$. The dashed lines represent the theoretical growth rate $\mathcal{O}(n^{(d+2k+1)/d})$. The empirical results show excellent agreement with the theory.}
        \label{fig:condition_number_growth}
\end{figure}

\section{Conclusion and Outlook}\label{sec_conclu}

In this work, we established sharp asymptotic estimates for the condition numbers of the mass and stiffness matrices associated with shallow ReLU$^k$ neural networks on the unit sphere $\SS^d$. Specifically, we proved that the condition numbers scale as
\[
\kappa(M) \simeq n^{\frac{d+2k+1}{d}}, \qquad \kappa(K_s) \simeq n^{\frac{d+2(k-s)+1}{d}}
\]
under antipodally quasi-uniform distribution of the parameters. These estimates reveal an explicit interplay between the network width $n$, activation smoothness $k$, and spatial dimension $d$.

Beyond providing precise growth rates, our results highlight a fundamental structural analogy between neural networks and classical finite element methods. In both settings, the mass and stiffness matrices govern the stability of the linear systems arising in projection-based approximation or least-squares optimization. Our analysis, grounded in spherical harmonics, Legendre expansions, and discrete summation by parts, offers a rigorous framework for understanding the spectral properties of overparameterized neural models on geometric domains.

The comparison with finite element methods further emphasizes a universal trade-off between approximation power and numerical stability, but also reveals an important distinction: in FEM, conditioning deteriorates when measured in higher Sobolev norms $\mathcal{H}^s$, while in linearized ReLU$^k$ networks the dependence on the derivative order $s$ is reversed due to the structure of the activation functions. Our comparison with finite element methods further reinforces the idea that higher approximation power inevitably leads to increased spectral instability, reflecting a universal principle in approximation theory and machine learning alike. Our results provide a rigorous explanation for the "curse of conditioning" that plagues methods like the shallow Ritz method in high dimensions. This highlights that for such methods to be viable in practice, the development of effective preconditioning strategies is not just an optimization but an essential requirement.

Several directions remain open. First, it would be interesting to extend the current analysis beyond the sphere $\SS^d$, for example to more general compact Riemannian manifolds or to discrete point clouds with approximate manifold structure. Second, the assumption of quasi-uniform parameter distributions could be relaxed to allow for clustered or learned parameters. Third, it would be valuable to study how these condition number bounds influence optimization dynamics and generalization behavior, particularly in deeper settings. Finally, connecting our condition number estimates with bounds on training error or capacity may provide insight into practical aspects of neural network training.

Overall, our results lay the mathematical groundwork for a precise understanding of conditioning and stability in structured neural architectures from the perspective of spectral analysis.

\bibliographystyle{abbrv}
	\bibliography{ref}
\end{document}